\documentclass[10pt]{amsart}

\usepackage{amsthm,amsfonts,amsmath,verbatim,mathrsfs,
amssymb,verbatim,extarrows,cite}
\usepackage[usenames]{color}
\usepackage{tikz,graphicx}
\usepackage{hyperref}

\numberwithin{equation}{section}
\numberwithin{figure}{section}
\newtheorem{theorem}{Theorem}[section]
\newtheorem{corollary}[theorem]{Corollary}
\newtheorem{proposition}[theorem]{Proposition}
\newtheorem{conjecture}[theorem]{Conjecture}
\newtheorem{lemma}[theorem]{Lemma}

\newcommand{\charge}{\textrm{c}}
\newcommand{\LHS}{\textup{LHS}}

\DeclareMathOperator{\shape}{shape}
\DeclareMathOperator{\End}{End}
\DeclareMathOperator{\Tab}{Tab}
\DeclareMathOperator{\sgn}{sgn}
\DeclareMathOperator{\eup}{e}
\DeclareMathOperator{\symp}{sp}
\DeclareMathOperator{\so}{so}
\DeclareMathOperator{\iup}{i}
\DeclareMathOperator{\odd}{o}

\renewcommand{\sc}{\scriptstyle}

\newcommand{\Rat}{\mathbb Q}
\newcommand{\Real}{\mathbb R}

\newcommand{\Z}{\mathbb Z}
\newcommand{\NN}{\Z_{+}}

\newcommand{\Complex}{\mathbb C}
\newcommand{\Symm}{\mathfrak{S}}

\newcommand{\abs}[1]{\lvert#1\rvert}
\newcommand{\la}{\lambda}
\newcommand{\La}{\Lambda}

\newcommand{\ip}[2]{\langle#1,#2\rangle}
\newcommand{\ipb}[2]{(#1|#2)}

\DeclareMathOperator{\lev}{lev}

\newcommand{\A}{\mathrm A}
\newcommand{\B}{\mathrm B}
\newcommand{\BC}{\mathrm{BC}}
\newcommand{\C}{\mathrm C}
\newcommand{\D}{\mathrm D}

\newcommand{\qbin}[2]{\genfrac{[}{]}{0pt}{}{#1}{#2}}
\newcommand{\gfrak}{\mathfrak{g}}
\newcommand{\hfrak}{\mathfrak{h}}
\DeclareMathOperator{\ch}{ch}
\DeclareMathOperator{\mult}{mult}

\newcommand{\rhob}{\bar{\rho}}
\newcommand{\Lab}{\bar{\La}}
\newcommand{\Phien}{{{_{1}\hspace{-20000sp}}\Phi_0}}

\begin{document}

\title[Hall--Littlewood polynomials]
{Hall--Littlewood polynomials and characters of affine Lie algebras}

\author{Nick Bartlett}

\author{S.~Ole Warnaar}

\address{School of Mathematics and Physics,
The University of Queensland, Brisbane, QLD 4072, Australia}

\thanks{Work supported by the Australian Research Council}


\begin{abstract}
The Weyl--Kac character formula gives a beautiful closed-form expression
for the characters of integrable highest-weight modules of Kac--Moody
algebras. 
It is not, however, a formula that is combinatorial in nature, 
obscuring positivity.
In this paper we show that the theory of Hall--Littlewood polynomials may 
be employed to prove Littlewood-type combinatorial formulas for the 
characters of certain highest weight modules of the affine Lie algebras 
$\C_n^{(1)}$, $\A_{2n}^{(2)}$ and $\D_{n+1}^{(2)}$. 
Through specialisation this yields generalisations for
$\B_n^{(1)}$, $\C_n^{(1)}$, $\A_{2n-1}^{(2)}$, $\A_{2n}^{(2)}$
and $\D_{n+1}^{(2)}$
of Macdonald's identities for powers of the
Dedekind eta-function. These generalised eta-function identities
include the Rogers--Ramanujan, Andrews--Gordon and G\"ollnitz--Gordon 
$q$-series as special, low-rank cases.
\end{abstract}

\maketitle

\section{Introduction}

Let $\gfrak$ be a symmetrisable Kac--Moody Lie algebra and
$\hfrak^{\ast}$ the dual of the Cartan subalgebra of $\gfrak$.
If $P_{+}$ denotes the set of dominant integral weights,
then the character of an irreducible $\gfrak$-module
$V(\La)$ of highest weight $\La\in P_{+}$ is defined as
\[
\ch V(\La)=\sum_{\mu\in\hfrak^{\ast}} \dim(V_{\mu}) \eup^{\mu}.
\]
Here $\eup^{\mu}$ is a formal exponential and $\dim(V_{\mu})$ 
the dimension of the weight space $V_{\mu}$ in the weight-space 
decomposition of $V(\La)$.
The celebrated Weyl--Kac formula gives a closed-form formula for
the character of $V(\La)$ as \cite{Kac74,Kac90}
\begin{equation}\label{Eq_Weyl-Kac}
\ch V(\La)=
\frac{\sum_{w\in W}\sgn(w) \eup^{w(\La+\rho)-\rho}}
{\prod_{\alpha>0}(1-\eup^{-\alpha})^{\mult(\alpha)}},
\end{equation}
where $W$ is the Weyl group of $\gfrak$, $\sgn(w)$ the signature of
$w\in W$ and $\rho$ the Weyl vector.
The product over $\alpha>0$ is shorthand for a product over
the set of positive roots of $\gfrak$, and
$\mult(\alpha)$ is the dimension of the root space corresponding to $\alpha$.
If $\gfrak$ is of classical type, then $\mult(\alpha)=1$ and 
\eqref{Eq_Weyl-Kac} simplifies to the Weyl character formula.

One feature of characters not evident from the Weyl--Kac formula
is positivity, and a natural question is whether other 
closed-form expressions exist that are manifestly positive.
The purpose of this paper is to show that for the affine
Lie algebras $\C_n^{(1)}$, $\A_{2n}^{(2)}$ and $\D_{n+1}^{(2)}$, 
there is an affirmative answer to this question.
The main player in these manifestly-positive formulas
is the modified Hall--Littlewood polynomial $Q'_{\mu}$ indexed
by the partition (as opposed to weight) $\mu$.
The $Q'_{\mu}$ is a symmetric function with nonnegative coefficients 
in $\Z[q]$ admitting a purely combinatorial
description. For example, for $x=(x_1,\dots,x_n)$,
\begin{equation}\label{Eq_Qpcharge}
Q'_{\mu}(x;q)=\sum_{T\in\Tab(\cdot,\mu)} q^{\charge(T)} s_{\shape(T)}(x)=
\sum_{\la} K_{\la\mu}(q) s_{\la}(x),
\end{equation}
where $\Tab(\la,\mu)$ is the set of semistandard Young tableaux of 
shape $\la$ and weight $\mu$, $s_{\la}(x)$ is the classical Schur function, 
$\charge(T)$ the Lascoux--Sch\"utzenberger charge \cite{LS78} and
$K_{\la\mu}=\sum_{T\in\Tab(\la,\mu)} q^{\charge(T)}$ the Kostka--Foulkes 
polynomial \cite{DLT94,Macdonald95}.

To give an example of the type of results obtained in this paper 
we need some more notation. 
For $\la$ a partition, let $\abs{\la}=\sum_{i\geq 1} \la_i$ and 
$b_{\la}(q)=\prod_{i\geq 1} (q)_{m_i(\la)}$, where $m_i(\la)$ is the 
multiplicity of parts of size $i$ in $\la$ and $(q)_k=(1-q)\cdots(1-q^k)$.
For example, if $\la=(4,4,2,1,1,1)=\big(4^22^11^3\big)$ 
then $b_{\la}(q)=(q)_2(q)_1(q)_3$. If all parts of $\la$ are even 
we say that $\la$ is even.
Now define a second modified Hall--Littlewood polynomial $P'_{\la}$ by
\begin{equation}\label{Eq_Pp}
P'_{\la}(x;q)=Q'_{\la}(x;q)/b_{\la}(q),
\end{equation}
so that its coefficients are in $\Rat(q)$ with nonnegative power
series expansion.
For $\gfrak$ one of $\C_n^{(1)}$, $\A_{2n}^{(2)}$ and
$\D^{(2)}_{n+1}$ with labelling of the Dynkin diagram as shown in 
Figure~\ref{Fig_1}, let
$\{\alpha_0,\dots,\alpha_n\}$, $\{\La_0,\dots,\La_n\}$ and
$\{a_0,\dots,a_n\}$ be the set of simple roots, fundamental weights and
marks of $\gfrak$, and let
$\delta=\sum_{i=0}^n a_i \alpha_i$ be the null root.
Finally, for $x=(x_1,\dots,x_n)$ define
$f\big(x^{\pm}\big):=f(x_1,x_1^{-1},\dots,x_n,x_n^{-1})$.

\begin{theorem}\label{Thm_Cn-character-formula}
Fix a nonnegative integer $m$ and let
\[
q=\eup^{-\delta}
\quad\text{and}\quad
x_i=\eup^{-\alpha_i-\cdots-\alpha_{n-1}-\alpha_n/2}.
\]
Then, for $\gfrak=\C_n^{(1)}$ and $\La=m\La_0$,
\begin{subequations}\label{Eq_CnA2n2-char}
\begin{align}\label{Eq_Cn-char}
\eup^{-\La} \ch V(\La)&=
\sum_{\substack{\la \textup{ even} \\[1.5pt] \la_1\leq 2m}}
q^{\abs{\la}/2} P'_{\la}\big(x^{\pm};q\big) 
\intertext{and, for $\gfrak=\A_{2n}^{(2)}$ and $\La=2m\La_0$,}
\eup^{-\La} \ch V(\La)&=\sum_{\substack{\la \\[1.5pt] \la_1\leq 2m}}
q^{\abs{\la}/2} P'_{\la}\big(x^{\pm};q\big).
\label{Eq_A2n2-char}
\end{align}
\end{subequations}
\end{theorem}

We note the remarkable similarity between \eqref{Eq_CnA2n2-char}
and the following well-known Littlewood-type character identities for the 
classical groups $\C_n$ and $\B_n$:
\begin{align*}
(x_1\cdots x_n)^m \symp_{2n,(m^n)}(x)
&=\sum_{\substack{\la \textup{ even} \\[1.5pt] \la_1\leq 2m}}
s_{\la}(x) \\
(x_1\cdots x_n)^m \so_{2n+1,(m^n)}(x)
&=\sum_{\substack{\la \\[1.5pt] \la_1\leq 2m}}
s_{\la}(x),
\end{align*}
where $\symp_{2n,\la}$ and $\so_{2n+1,\la}$ are the symplectic and 
odd orthogonal Schur functions (see \eqref{Eq_SchurBC} below),
and where the second identity also allows for half-integer $m$.
These identities have played an important role in the theory of plane 
partitions, see e.g., \cite{Bressoud99,Desarmenien86,Krattenthaler98,Macdonald95,Okada98,Proctor90,Stembridge90,Stembridge90b}.

The map $\exp(-\alpha_i)\mapsto 1$ for all $1\leq i\leq n$ (i.e.,
$x_i\mapsto 1$) is known as the basic specialisation \cite{Kac90}. 
Applied to Theorem~\ref{Thm_Cn-character-formula},
where on the left the Weyl--Kac expression \eqref{Eq_Weyl-Kac}
is used, leads to
the following generalisations of Macdonald's $\C_n^{(1)}$ 
and $\A_{2n}^{(2)}$ (or affine $\BC_n$) eta-function identities 
\cite{Macdonald72}. Let 
\begin{equation}\label{Eq_chi-B}
\chi_{\B}(v):=\prod_{i=1}^n v_i \prod_{1\leq i<j\leq n} (v_i^2-v_j^2),
\qquad\quad
\chi_{\B}(v/w)=\chi_{\B}(v)/\chi_{\B}(w),
\end{equation}
and $(a)_{\infty}=(a;q)_{\infty}=(1-a)(1-aq)(1-aq^2)\cdots$.

\begin{corollary}
Let $m$ be a nonnegative integer and $\rho=(n,\dots,2,1)$ the $\C_n$ 
Weyl vector. Then
\begin{subequations}
\begin{equation}\label{Eq_MD-C}
\frac{1}{(q)_{\infty}^{2n^2+n}}
\sum \chi_{\B}(v/\rho)
q^{\frac{\|v\|^2-\|\rho\|^2}{4(m+n+1)}}
=\sum_{\substack{\la \textup{ even} \\[1.5pt] \la_1\leq 2m}}
q^{\abs{\la}/2} P'_{\la}(\underbrace{1,\dots,1}_{2n \textup{ times}};q),
\end{equation}
where the sum on the left is over $v\in\Z^n$ 
such that $v\equiv \rho\pmod{2m+2n+2}$, and
\begin{multline}\label{Eq_MD-BC}
\frac{1}{(q^{1/2};q^{1/2})_{\infty}^{2n}(q^2;q^2)_{\infty}^{2n}
(q)_{\infty}^{2n^2-3n}}
\sum \chi_{\B}(v/\rho)
q^{\frac{\|v\|^2-\|\rho\|^2}{2(2m+2n+1)}} \\
=\sum_{\substack{\la \\[1.5pt] \la_1\leq 2m}}
q^{\abs{\la}/2} P'_{\la}(\underbrace{1,\dots,1}_{2n \textup{ times}};q),
\end{multline}
\end{subequations}
where the sum on the left is over $v\in\Z^n$ 
such that $v\equiv \rho\pmod{2m+2n+1}$.
\end{corollary}

Theorem~\ref{Thm_Cn-character-formula} and similar combinatorial 
character formulae such as \eqref{Eq_A2n2-char2} 
(for the $\A_{2n}^{(2)}$-module $V(m\Lambda_n)$)
and \eqref{Eq_twistedD}
(for the $\D_{n+1}^{(2)}$-module $V(2m\Lambda_0)$)
only deal with a restricted set of weight $\Lambda\in P_{+}$.
We believe however that the type of results obtained in this
paper hold more generally.
For example, computer experiments suggest that for $\C_n^{(1)}$ we have
\[
\eup^{-\La_1} \ch V(\La_1)=x_1 \sum_{k=0}^{\infty}
\frac{q^k}{(q)_k}\, Q'_{(2^k1)}\big(x^{\pm};q\big).
\]

\medskip

The remainder of this paper is organised as follows.
In the next section, after reviewing some standard material from the
theory of affine Kac--Moody algebras, we rewrite the Weyl--Kac
formula \eqref{Eq_Weyl-Kac} for 
$\gfrak=\C_n^{(1)}$, $\A_{2n}^{(2)}$ and $\D_{n+1}^{(2)}$ as a sum over
symplectic or odd orthogonal Schur functions.
In Section~\ref{Sec_HL} we use Jing's vertex operators to
prove a new basic hypergeometric formula for modified 
Hall--Littlewood polynomials $P'_{\la}$,
and apply this to obtain a Littlewood-type 
summation formula for modified Hall--Littlewood polynomials. 
We further connect these results with Rogers--Ramanujan and
Nahm--Zagier-type $q$-series.
In Section~\ref{Sec_Cn-Andrews} we employ the Milne--Lilly Bailey lemma 
for the $\C_n$ root system to prove a $\C_n$ analogue of Andrews' well-known
multiple series transformation. 
Then, in Section~\ref{Sec_5}, it is shown that after specialisation,
and a somewhat intricate limiting procedure, one side of the
$\C_n$ Andrews transformation corresponds to 
certain characters in their Weyl--Kac representation. Furthermore,
applying the Littlewood-type summation formula from Section~\ref{Sec_HL}
we show that the other side is expressible in terms of $P'_{\la}$, 
resulting in a proof of our combinatorial character formulas.
In Section~\ref{Sec_dedekind} we provide a compendium to Macdonald's 
famous list of identities for powers of the Dedekind eta-function, extending
his identities for affine $\B_n$, $\C_n$, $\D_n$ and $\BC_n$ to infinite 
families of such identities. 
Finally, in Section~\ref{Sec_Conclusion}, we make some concluding remarks
in response to questions posed by one of the referees.
This includes a brief discussion of an alternative approach 
to combinatorial character identities recently 
developed by Eric Rains and the second author.

\subsection*{Acknowledgements}
We thank both referees for their constructive comments and interesting
questions.

\section{Affine Kac--Moody algebras}\label{Sec_Weyl-Kac}

In order to prove the main results of this paper, such as 
Theorem~\ref{Thm_Cn-character-formula}, we require a simple
rewriting of the Weyl--Kac formula \eqref{Eq_Weyl-Kac}
for $\gfrak$ one of $\C_n^{(1)}$, $\A_{2n}^{(2)}$ and $\D_{n+1}^{(2)}$
in terms of the odd orthogonal and symplectic Schur functions
\cite{Littlewood50}
\begin{subequations}\label{Eq_SchurBC}
\begin{align}
\label{Eq_odd-orthogonal}
\so_{2n+1,\la}(x)&=\frac{\det_{1\leq i,j\leq n} 
\big(x_i^{j-1-\la_j}-x_i^{2n-j+\la_j}\big)}
{\Delta_{\B}(x)}, \\
\symp_{2n,\la}(x)&=\frac{\det_{1\leq i,j\leq n} 
\big(x_i^{j-1-\la_j}-x_i^{2n-j+1+\la_j}\big)}
{\Delta_{\C}(x)}.
\label{Eq_symplectic}
\end{align}
\end{subequations}
Here $\Delta_{\B}$ and $\Delta_{\C}$ are the generalised Vandermonde products
\begin{align*}
\Delta_{\B}(x)&:=\prod_{i=1}^n (1-x_i)
\prod_{1\leq i<j\leq n} (x_i-x_j)(x_ix_j-1) \\
\Delta_{\C}(x)&:=\prod_{i=1}^n (1-x_i^2)
\prod_{1\leq i<j\leq n} (x_i-x_j)(x_ix_j-1).
\end{align*}
In Section~\ref{SubSec_WK} will give the full details of this 
rewrite for $\C_n^{(1)}$ and then state the remaining cases without proof.

First however, we need to recall some basic notions from the general
theory of affine Kac--Moody algebras. For more details and background
material we refer the reader to the monographs by Kac \cite{Kac90} 
and Wakimoto \cite{Wakimoto01}.

\subsection{General definitions and notation}

Let $\gfrak=\gfrak(A)$ be an affine Kac--Moody algebra with
generalised Cartan matrix $A=(a_{ij})_{i,j\in I}$,
$I:=\{0,1,\dots,n\}$. We are primarily interested in $\gfrak$ of type 
$\C_n^{(1)}~(n\geq 1)$, $\A_{2n}^{(2)}~(n\geq 1)$ and 
$\D_{n+1}^{(2)}~(n\geq 2)$, although most of this section
applies to arbitrary type.
Let $\hfrak$ and $\hfrak^{\ast}$ be the $(n+2)$-dimensional Cartan subalgebra 
and its dual. Fix linearly independent elements 
$\alpha_0^{\vee},\dots,\alpha_n^{\vee}$ and 
$\alpha_0,\dots,\alpha_n$ of $\hfrak$ and $\hfrak^{\ast}$,
called simple coroots and simple roots, such that
$\ip{\alpha_i^{\vee}}{\alpha_j}=a_{ij}$.
Extend the above to a basis of $\hfrak$ and $\hfrak^{\ast}$ by choosing the
additional elements $d\in\hfrak$ and $\La_0\in\hfrak^{\ast}$ such that
$\ip{\alpha_i^{\vee}}{\La_0}=\ip{d}{\alpha_i}=\delta_{i,0}$ and 
$\ip{d}{\La_0}=0$.
The marks and comarks (also known as labels and colabels) 
$a_0,\dots,a_n$ and $a_0^{\vee},\dots,a_n^{\vee}$
are positive integers, uniquely determined by 
$\sum_{i\in I} a_{ij}a_j=\sum_{i\in I} a_i^{\vee}a_{ij}=0$
such that 
\[
\gcd(a_0,\dots,a_n)=\gcd(a_0^{\vee},\dots,a_n^{\vee})=1.
\]
The sum of the marks and comarks are known as the
Coxeter and dual Coxeter number respectively,
$h=\sum_{i\in I} a_i$ and $h^{\vee}=\sum_{i\in I} a_i^{\vee}$.
The Dynkin diagrams of the three infinite series of interest
are given in Figure~\ref{Fig_1},
together with a labelling of the vertices by simple roots $\alpha_i$ and
marks $a_i$.

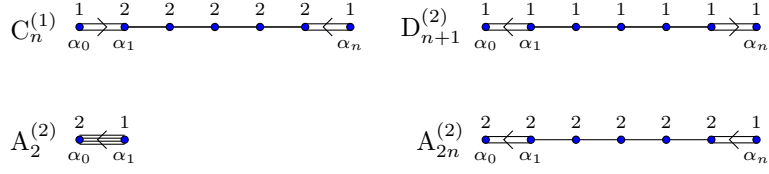
\begin{figure}[th]
\begin{center}
\begin{tikzpicture}[scale=0.6]
\draw (-1,-2.5) node {$\C_n^{(1)}$};
\draw (0,-2.43)--(1,-2.43);
\draw (0,-2.57)--(1,-2.57);
\draw (0.4,-2.3)--(0.6,-2.5)--(0.4,-2.7);
\draw (1,-2.5)--(5,-2.5);
\draw (1,-2.5)--(5,-2.5);
\draw (5.6,-2.3)--(5.4,-2.5)--(5.6,-2.7);
\draw (5,-2.43)--(6,-2.43);
\draw (5,-2.57)--(6,-2.57);
\foreach \x in {0,...,6} \draw[fill=blue] (\x,-2.5) circle (0.08cm);
\draw (0,-2.9) node {$\sc \alpha_0$};
\draw (1,-2.9) node {$\sc \alpha_1$};
\draw (6,-2.9) node {$\sc \alpha_n$};
\draw (0,-2.1) node {$\sc 1$};
\draw (1,-2.1) node {$\sc 2$};
\draw (2,-2.1) node {$\sc 2$};
\draw (3,-2.1) node {$\sc 2$};
\draw (4,-2.1) node {$\sc 2$};
\draw (5,-2.1) node {$\sc 2$};
\draw (6,-2.1) node {$\sc 1$};
\draw (7.8,-2.5) node {$\D_{n+1}^{(2)}$};
\draw (9,-2.43)--(10,-2.43);
\draw (9,-2.57)--(10,-2.57);
\draw (9.6,-2.3)--(9.4,-2.5)--(9.6,-2.7);
\draw (10,-2.5)--(14,-2.5);
\draw (10,-2.5)--(14,-2.5);
\draw (14.4,-2.3)--(14.6,-2.5)--(14.4,-2.7);
\draw (14,-2.43)--(15,-2.43);
\draw (14,-2.57)--(15,-2.57);
\foreach \x in {9,...,15} \draw[fill=blue] (\x,-2.5) circle (0.08cm);
\draw (9,-2.9) node {$\sc \alpha_0$};
\draw (10,-2.9) node {$\sc \alpha_1$};
\draw (15,-2.9) node {$\sc \alpha_n$};
\draw (9,-2.1) node {$\sc 1$};
\draw (10,-2.1) node {$\sc 1$};
\draw (11,-2.1) node {$\sc 1$};
\draw (12,-2.1) node {$\sc 1$};
\draw (13,-2.1) node {$\sc 1$};
\draw (14,-2.1) node {$\sc 1$};
\draw (15,-2.1) node {$\sc 1$};
\draw (-1,-5) node {$\A_2^{(2)}$};
\draw (0,-4.9)--(1,-4.9);
\draw (0,-4.966)--(1,-4.966);
\draw (0,-5.033)--(1,-5.033);
\draw (0,-5.1)--(1,-5.1);
\draw (0.6,-4.8)--(0.4,-5)--(0.6,-5.2);
\foreach \x in {0,...,1} \draw[fill=blue] (\x,-5) circle (0.08cm);
\draw (0,-5.4) node {$\sc \alpha_0$};
\draw (1,-5.4) node {$\sc \alpha_1$};
\draw (0,-4.6) node {$\sc 2$};
\draw (1,-4.6) node {$\sc 1$};
\draw (8,-5) node {$\A_{2n}^{(2)}$};
\draw (9,-4.93)--(10,-4.93);
\draw (9,-5.07)--(10,-5.07);
\draw (9.6,-4.8)--(9.4,-5)--(9.6,-5.2);
\draw (10,-5)--(14,-5);
\draw (14.6,-4.8)--(14.4,-5)--(14.6,-5.2);
\draw (14,-4.93)--(15,-4.93);
\draw (14,-5.07)--(15,-5.07);
\foreach \x in {9,...,15} \draw[fill=blue] (\x,-5) circle (0.08cm);
\draw (9,-5.4) node {$\sc \alpha_0$};
\draw (10,-5.4) node {$\sc \alpha_1$};
\draw (15,-5.4) node {$\sc \alpha_n$};
\draw (9,-4.6) node {$\sc 2$};
\draw (10,-4.6) node {$\sc 2$};
\draw (11,-4.6) node {$\sc 2$};
\draw (12,-4.6) node {$\sc 2$};
\draw (13,-4.6) node {$\sc 2$};
\draw (14,-4.6) node {$\sc 2$};
\draw (15,-4.6) node {$\sc 1$};
\end{tikzpicture}
\end{center}
\caption{\small The Dynkin diagrams of the three infinite series of
affine Lie algebras of interest, together with a labelling of vertices 
by simple roots and by the marks $a_0,\dots,a_n$. 
$\C_n^{(1)}$ and $\D_{n+1}^{(2)}$ are dual
and the comarks of $\gfrak$ are the marks of its dual.
The comarks of $\A_{2n}^{(2)}$ are its marks read in reverse order.}
\label{Fig_1}
\end{figure}

We now fix what is known as the standard non-degenerate bilinear form 
on $\hfrak$ by setting  
\[
\ipb{\alpha_i^{\vee}}{\alpha_j^{\vee}}=\frac{a_j}{a_j^{\vee}}\,a_{ij},
\qquad 
\ipb{\alpha_i^{\vee}}{d}=a_0\delta_{i,0},\qquad
\ipb{d}{d}=0.
\]
We adopt the natural identification of $\hfrak$ with $\hfrak^{\ast}$
by identifying $d$ with $a_0\La_0$ and $\alpha_i^{\vee}$ with 
$a_i\alpha_i/a_i^{\vee}$.
Then
\[
\ipb{\alpha_i}{\alpha_j}=\frac{a_i^{\vee}}{a_i}\,a_{ij},
\qquad 
\ipb{\alpha_i}{\La_0}=\frac{1}{a_0}\,\delta_{i,0},\qquad
\ipb{\La_0}{\La_0}=0.
\]

Before we turn to the Weyl--Kac formula a few more definitions are needed.
The null root or fundamental imaginary root $\delta$
is defined as $\delta=\sum_{i\in I} a_i\alpha_i$.
Then $\hfrak^{\ast}=\Complex\Lambda_0\oplus\bar{\hfrak}^{\ast}\oplus\Complex\delta$
where $\overline{\hfrak}^{\ast}=\sum_{i\in \bar{I}} \Complex \alpha_i$ for 
$\bar{I}:=\{1,2,\dots,n\}$ is the finite part of $\hfrak^{\ast}$.
We complement $\Lambda_0$ to a full set of 
fundamental weights $\La_0,\dots,\La_n\in\hfrak^{\ast}$ by
\[
\ip{\La_i}{\alpha_j^{\vee}}=\delta_{ij},\qquad
\ip{\La_i}{d}=0.
\]
The Weyl vector $\rho\in\hfrak^{\ast}$ is given by
$\ip{\rho}{\alpha_i^{\vee}}=1$ for all $i\in I$ and $\ip{\rho}{d}=0$. 
If $K$ is the canonical central element
$K=\sum_{i\in I} a_i^{\vee} \alpha_i^{\vee}$ then the level $\lev(\la)$ of
$\la\in\hfrak^{\ast}$ is given by $\lev(\la)=\ip{\la}{K}$. Note that $\lev(\La_0)=1$
and $\lev(\rho)=h^{\vee}$.

The root and coroot lattices $Q$ and $Q^{\vee}$
are defined by the integer span of the
simple roots and simple coroots respectively.
Similarly, $\overline{Q}=\sum_{i\in\bar{I}} \Z \alpha_i$ and 
$\overline{Q}^{\vee}=\sum_{i\in\bar{I}} \Z \alpha^{\vee}_i$.
One further lattice that will play an important role is
\begin{equation}\label{Eq_M-lattice}
M=\begin{cases}
\overline{Q}^{\vee} & \text{if $\gfrak=\mathrm{X}_n^{(1)}$ or $\gfrak=\A_{2n}^{(2)}$}, \\
\overline{Q} & \text{otherwise}.
\end{cases}
\end{equation}
To conclude our string of definitions we let $P_{+}$ denote
the set of dominant integral weights
\[
P_{+}=\{\la\in\hfrak^{\ast}:~\ip{\la}{\alpha_i^{\vee}}\in\NN:~
\text{for all $i\in I$}\},
\]
where throughout this paper, $\NN$ denotes the set of
nonnegative integers.

\subsection{The Weyl--Kac formula}\label{SubSec_WK}

To achieve the desired rewriting of the Weyl--Kac formula
we first follow Kac and Peterson \cite{KP84}.
Let $\overline{W}$ be the finite Weyl group corresponding to the Cartan matrix
$\bar{A}$ obtained from $A$ by deleting the zeroth row and column; 
$\bar{A}=(a_{ij})_{i,j\in\bar{I}}$. 
Then the affine Weyl group $W$ of $\gfrak$ is given by $W=\overline{W}\ltimes M$ 
with $M$ the lattice \eqref{Eq_M-lattice}.
This allows \eqref{Eq_Weyl-Kac} to be restated as
\begin{multline}\label{WK}
\eup^{-\La}\ch V(\La)=
\prod_{\alpha>0}(1-\eup^{-\alpha})^{-\mathrm{mult(\alpha)}}\\
\times \sum_{\gamma\in M}
\sum_{w\in\overline{W}} \sgn(w) 
q^{\frac{1}{2}\kappa \ipb{\gamma}{\gamma}-\ipb{\gamma}{w(\Lab+\rhob)}}
\eup^{-\kappa\gamma+w(\Lab+\rhob)-\Lab-\rhob},
\end{multline}
where $\kappa=\lev(\La+\rho)=\lev(\La)+h^{\vee}$, $q=\exp(-\delta)$
and where $\bar{\lambda}$ again denotes the finite part.

Next we focus on $\gfrak=\C_n^{(1)}$ with generalised Cartan matrix
$A$ given by the tridiagonal matrix with $d_{-1}=(-2,-1,\dots,-1)$,
$d_0=(2,\dots,2)$ and $d_1=(-1,\dots,-1,-2)$.
The set of positive roots $\Delta_{+}$ consists of the disjoint subsets
of positive imaginary and positive real roots, given by
\[
\Delta_{+}^{\text{im}}=\big\{m\delta:~m\in\NN\setminus\{0\}\big\},
\]
each root occurring with multiplicity $n$, and 
\[
\Delta_{+}^{\text{re}}=
\Bigg\{m\delta+\alpha:~\alpha\in\bar{\Delta},\;
m\in\begin{cases} \Z_{+} & \text{if $\alpha\in\bar{\Delta}_{+}$} \\
\Z_{+}\setminus\{0\} & \text{otherwise} \end{cases}
\Bigg\},
\]
of multiplicity $1$. Here $\bar{\Delta}$ is the root system of $\gfrak(\bar{A})$
with base $\overline{\Pi}$.
In terms of the standard Euclidean description\footnote{We 
deviate from Kac's convention that $\ipb{\alpha}{\alpha}=2$
for $\alpha$ a long root. This comes at the cost of introducing the
factor $1/2$ in $\ipb{\epsilon_i}{\epsilon_j}=\delta_{ij}/2$ but avoids
the occurrence of $\sqrt{2}$ in some of our formulae.}
of $\overline{\Pi}$ and
$\bar{\Delta}_{+}=\bar{\Delta}_{s,+}\cup\bar{\Delta}_{\ell,+}$ we have
\[
\overline{\Pi}=\{\alpha_1,\dots,\alpha_n\}=
\{\epsilon_1-\epsilon_2,\dots,\epsilon_{n-1}-\epsilon_n,2\epsilon_n\}
\]
and
\[
\bar{\Delta}_{s,+}=\{\epsilon_i\pm \epsilon_j:~1\leq i<j\leq n\},
\qquad
\bar{\Delta}_{\ell,+}=\{2\epsilon_i:~1\leq i\leq n\}.
\]
Setting $x_i=\exp(-\epsilon_i)$ 
we thus get
\[
\prod_{\alpha>0}(1-\eup^{-\alpha})^{\mult(\alpha)}
=(q)_{\infty}^n\, \Delta_{\C}(x)
\prod_{i=1}^n x_i^{1-i} (qx_i^{\pm 2})_{\infty}
\prod_{1\leq i<j\leq n} (qx_i^{\pm}x_j^{\pm})_{\infty},
\]
where $(au^{\pm})_{\infty}=(au,au^{-1})_{\infty}$ and
$(au^{\pm}v^{\pm})_{\infty}=
(auv,auv^{-1},au^{-1}v,au^{-1}v^{-1})_{\infty}$
for
$(a_1,\dots,a_k)_{\infty}=(a_1)_{\infty}\cdots(a_k)_{\infty}$.

Next we consider the numerator of \eqref{WK}.
The lattice $M=\overline{Q}^{\vee}$ is spanned by
\[
2\{\epsilon_1-\epsilon_2,\dots,\epsilon_{n-1}-\epsilon_n,\epsilon_n\},
\]
i.e., $M$ is the classical $\B_n$ root lattice scaled by a factor
of two
\[
M=\bigg\{2\sum_{i=1}^n r_i \epsilon_i:~(r_1,\dots,r_n)\in\Z^n\bigg\}.
\]
We also use that $\overline{W}$ is the hyperoctahedral
group (or the group of signed permutations) $\overline{W}=
\Symm_n \ltimes (\Z/2\Z)^n$ with natural action on $\Real^n$, see
e.g., \cite{Humphreys78}.
Finally, for $\La=c_0\Lambda_0+\cdots+c_n\Lambda_n\in P_{+}$
define the partition $\la=(\la_1,\dots,\la_n)$ by
$\la_i=c_i+\cdots+c_n$.
Hence, since $\bar{\La}_i=\epsilon_1+\cdots+\epsilon_i$,
we have $\bar{\La}+\bar{\rho}=\sum_{i=1}^n (\la_i+\rho_i)\epsilon_i$,
where $\rho_i:=n-i+1$.
Also note that
\[
\kappa=\sum_{i=0}^n a_i^{\vee}(c_i+1)=h^{\vee}+c_0+\cdots+c_n=n+1+c_0+\la_1.
\]
Therefore, the double sum in \eqref{WK} yields
\begin{multline*}
\sum_{r\in\Z^n} \sum_{w\in\overline{W}} \sgn(w)
\prod_{i=1}^n q^{\kappa r_i^2
-2r_i \sum_{j=1}^n(\la_j+\rho_j)\ipb{\epsilon_i}{w(\epsilon_j)}} 
x_i^{2\kappa r_i+\la_i+\rho_i} w\big(x_i^{-\la_i-\rho_i}\big) \\
=\sum_{r\in\Z^n}\prod_{i=1}^n q^{\kappa r_i^2} x_i^{2\kappa r_i+\la_i+\rho_i}
\sum_{w\in\overline{W}} \sgn(w) w\Big(\prod_{i=1}^n y_i^{-\la_i-\rho_i}\Big),
\end{multline*}
where $y_i:=x_iq^{r_i}$.
By \eqref{Eq_symplectic} the sum over $\overline{W}$ is given by
\[
\Delta_{\C}(y) \symp_{2n,\la}(y) \prod_{i=1}^n y_i^{-n}
\]
so that we obtain the next lemma.
\begin{lemma}[$\C_n^{(1)}$ character formula]\label{Lem_Cn-WK}
For $q=\exp(-\delta)$, $\la=(\la_1,\dots,\la_n)$ a partition and
\begin{subequations}\label{Eq_Cn-char-La}
\begin{gather}
\La=c_0\La_0+(\la_1-\la_2)\La_1+\cdots+(\la_{n-1}-\la_n)\La_{n-1}+\la_n\La_n
\in P_+, \\[2mm]
x_i=\eup^{-\alpha_i-\cdots-\alpha_{n-1}-\alpha_n/2},
\end{gather}
\end{subequations}
we have
\begin{multline}\label{Eq_char-Cn}
\eup^{-\La} \ch V(\La)=
\frac{1}{(q)_{\infty}^n\prod_{i=1}^n (qx_i^{\pm 2})_{\infty}
\prod_{1\leq i<j\leq n} (qx_i^{\pm}x_j^{\pm})_{\infty}} \\[2mm]
\times
\sum_{r\in\Z^n}\frac{\Delta_{\C}(xq^r)}{\Delta_{\C}(x)}
\prod_{i=1}^n q^{\kappa r_i^2-nr_i} x_i^{2\kappa r_i+\la_i} 
\symp_{2n,\la}\big(xq^r\big),
\end{multline}
where
$\kappa=n+1+c_0+\la_1$.
\end{lemma}

In much the same way we can rewrite the other characters of interest.

\begin{lemma}[$\A_{2n}^{(2)}$ character formula, I]\label{Lem_A2n2-WK-1}
With the same assumptions as in Lemma~\ref{Lem_Cn-WK},
\begin{multline}\label{Eq_I}
\eup^{-\Lambda} \ch V(\Lambda)= 
\frac{1}{(q)_{\infty}^n \prod_{i=1}^n 
(q^{1/2}x_i^{\pm})_{\infty}(q^2x_i^{\pm 2};q^2)_{\infty}
\prod_{1\leq i<j\leq n}(qx_i^{\pm}x_j^{\pm})_{\infty}} \\[2mm]
\times \sum_{r\in\Z^n}
\frac{\Delta_{\C}(xq^r)}{\Delta_{\C}(x)}
\prod_{i=1}^n q^{\frac{1}{2}\kappa r_i^2-nr_i} x_i^{\kappa r_i+\la_i} \symp_{2n,\la}(xq^r),
\end{multline}
where $\kappa=2n+1+c_0+2\la_1$.
\end{lemma}
Viewing the Dynkin diagram of $\A_{2n}^{(2)}$ in a mirror 
leads to an alternative, $\B$-type expression for the above character.
\begin{lemma}[$\A_{2n}^{(2)}$ character formula, II]\label{Lem_A2n2-WK-2}
For $q=\exp(-\delta)$, $\mu=(\mu_1,\dots,\mu_n)$ a partition or 
half-partition, and
\begin{gather*}
\La=2\mu_n\La_0+(\mu_{n-1}-\mu_n)\La_1+\cdots+(\mu_1-\mu_2)\La_{n-1}+c_n\La_n
\in P_+, \\[2mm]
y_i=\eup^{-\alpha_0-\cdots-\alpha_{n-i}},
\end{gather*}
\textup{(}so that $y_i=q^{1/2}x_{n-i+1}^{-1}$ and 
$\mu_i=c_0/2+\la_1-\la_{n-i+1}$ compared to \eqref{Eq_I}\textup{)},
\begin{multline*}
\eup^{-\Lambda} \ch V(\Lambda)= 
\frac{1}{(q)_{\infty}^n \prod_{i=1}^n 
(qy_i^{\pm})_{\infty}(qy_i^{\pm 2};q^2)_{\infty}
\prod_{1\leq i<j\leq n}(qy_i^{\pm}y_j^{\pm})_{\infty}} \\[2mm]
\times \sum_{r\in\Z^n}
\frac{\Delta_{\B}(yq^r)}{\Delta_{\B}(y)}
\prod_{i=1}^n q^{\frac{1}{2}\kappa r_i^2-(n-\frac{1}{2})r_i} 
y_i^{\kappa r_i+\mu_i} \so_{2n+1,\mu}(yq^r),
\end{multline*}
where $\kappa=2n+1+2c_n+2\mu_1$.
\end{lemma}
Here a half-partition $\mu=(\mu_1,\mu_2,\dots,\mu_n)$ is a sequence of 
weakly decreasing positive numbers such that $\mu_i+1/2\in\Z$ for all $i$.

\begin{lemma}[$\D_{n+1}^{(2)}$ character formula]\label{Lem_Dn2-WK}
For $q=\exp(-\delta)$, $\la=(\la_1,\dots,\la_n)$ a partition 
or half-partition, and
\begin{subequations}\label{Eq_Bn-char}
\begin{gather}
\La=c_0\La_0+(\la_1-\la_2)\La_1+\cdots+(\la_{n-1}-\la_n)\La_{n-1}+2\la_n\La_n
\in P_+, \\[2mm]
x_i=\eup^{-\alpha_i-\cdots-\alpha_n},
\end{gather}
\end{subequations}
we have
\begin{multline*}
\eup^{-\Lambda} \ch V(\Lambda)=
\frac{1}{(q^2;q^2)_{\infty}^{n-1}(q)_{\infty}
\prod_{i=1}^n (qx_i^{\pm})_{\infty}
\prod_{1\leq i<j\leq n} (q^2x_i^{\pm}x_j^{\pm};q^2)_{\infty}} \\[2mm]
\times \sum_{r\in\Z^n} \frac{\Delta_{\B}(xq^{2r})}{\Delta_{\B}(x)}
\prod_{i=1}^n q^{\kappa r_i^2-(2n-1)r_i} 
x_i^{\kappa r_i+\la_i} \so_{2n+1,\la}\big(xq^{2r}\big),
\end{multline*}
where $\kappa=2n+c_0+2\la_1$.
\end{lemma}

\section{Modified Hall--Littlewood polynomials}\label{Sec_HL}

\subsection{Preliminaries}

The Hall--Littlewood polynomials are an important family of 
symmetric functions generalising the well-known Schur functions.
Our main interest will be the modified Hall--Littlewood
polynomials, for which we shall give a new, closed-form formula
as a multiple basic hypergeometric series. It is this formula 
that will ultimately allow us to express characters of affine Lie algebras 
in terms of modified Hall--Littlewood polynomials. 

For standard notation and terminology from the theory of partitions and
symmetric functions we refer the reader to \cite{Macdonald95}.

Fix a positive integer $n$.
For a partition $\la$ of length $l(\la)\leq n$
let $m_0(\la)=n-l(\la)$ and $m_i(\la)$ for $i\geq 1$ 
the multiplicity of parts of size $i$.
Define $v_{\la}(q)=\prod_{i\geq 0} (q)_{m_i(\la)}/(1-q)^{m_i(\la)}$.
If $\Symm_n$ denotes the symmetric group on $n$ letters and
$\Symm_n^{\la}$ the stabilizer of $\la$, then 
$v_{\la}(q)$ may be identified as the Poincar\'e polynomial
$\sum_{w\in\Symm_n^{\la}} t^{\ell(w)}$.
For $x=(x_1,\dots,x_n)$ the Hall--Littlewood polynomial 
$P_{\la}$ is the symmetric function \cite{Macdonald95}
\[
P_{\la}(x;q)=\frac{1}{v_{\la}(q)}
\sum_{w\in\Symm_n} 
w\bigg(x^\la\prod_{i<j}\frac{x_i-qx_j}{x_i-x_j}\bigg) 
=\sum_{w\in\Symm_n/\Symm_n^{\la}}
w\bigg(x^\la\prod_{\la_i>\la_j}\frac{x_i-qx_j}{x_i-x_j}\bigg).
\]
Here the symmetric group $\Symm_n$ acts on functions $f(x)$ 
by permuting the $x_i$.

The Hall--Littlewood polynomial $P_{\la}$ interpolates between
the Schur function $s_{\la}$ and the monomial symmetric 
function $m_{\la}$, corresponding to $q=0$ and $q=1$ respectively.
The $P_{\la}$, where $\la$ ranges over all partitions of length at most
$n$, form a basis of the ring of symmetric functions in $n$ variables.
There is a second Hall--Littlewood polynomial defined as
\begin{equation}\label{Eq_PQ}
Q_{\la}(x;q)=b_{\la}(q) P_{\la}(x;q),
\end{equation}
where $b_{\la}(q)=\prod_{i\geq 1} (q)_{m_i(\la)}=\prod_{i\geq 1}
(q)_{\la_i'-\la_{i+1}'}$, for $\la'$ the conjugate of $\la$.

The modified Hall--Littlewood polynomial $Q'_{\la}$ of equation
\eqref{Eq_Qpcharge} is a variant of $Q_{\la}$ which interpolates 
between the Schur function $s_{\la}$, obtained for $q=0$, and 
the complete symmetric function $h_{\la}$, obtained for $q=1$. 
Unlike the literature on the ordinary Hall--Littlewood polynomials, 
where the pair $P_{\la}$ and $Q_{\la}$ are usually given equal 
prominence, the polynomial $P'_{\la}$ defined in \eqref{Eq_Pp}
usually does not feature in work on the modified polynomials, see e.g.,
\cite{DLT94,Garsia92,GP92,Kirillov00,Lascoux05,Milne92}.
There are a number of reasons for this.
$Q'_{\la}$ has coefficients in $\Z[q]$, is Schur positive,
and has several combinatorial, representation theoretic 
and geometric interpretations.
$P'_{\la}$ on the other hand, has coefficients in $\Rat(q)$
and its $q\to 1$ limit does not exist due to $b_{\la}(1)=\delta_{\la,0}$.
Nonetheless, most of our results 
are simplest when expressed in terms of the $P'_{\la}$ and we will
use the two families of modified polynomials interchangeably.

Besides \eqref{Eq_Qpcharge} there exist numerous other descriptions of
the modified Hall--Littlewood polynomials, three of which will
be discussed below.
First of all, using the notation of $\la$-rings \cite{Lascoux01,Haglund08},
\[
Q'_{\la}(x;q)=Q_{\la}(x/(1-q);q)\quad\text{and}\quad
P'_{\la}(x;q)=P_{\la}(x/(1-q);q),
\]
where $x/(1-q)$ is shorthand for the infinite alphabet
obtained from $x$ be replacing each $x_i$ by $x_i,x_iq,x_iq^2,\dots$.
A second description of the modified Hall--Littlewood polynomials
uses the Hall inner product on the ring of symmetric functions,
defined by $\langle s_{\la},s_{\mu}\rangle=\delta_{\la\mu}$.
Then 
\[
\langle P_{\la},Q'_{\mu}\rangle=\langle P'_{\la},Q_{\mu}\rangle=
\delta_{\la\mu}.
\]

Finally, and most important for our purposes, the $Q'_{\la}$ can be
computed using Jing's $q$-Bernstein operators \cite{Jing91}
(see also \cite{Garsia92,Zabrocki00}).
Let $\Lambda$ be the ring of symmetric functions. For $f\in\Lambda$, 
denote by $f^{\perp}\in\End(\La)$ the operator (also 
known as Foulkes derivative) which acts as the adjoint of 
multiplication by $f$:
\[
\ip{f^{\perp}(g)}{h}=\ip{g}{fh} \quad\text{for $g,h\in\Lambda$}.
\]
For $m$ an integer the $q$-Bernstein operator $B_m=B_m(x;q)$ 
is defined as
\[
B_m=\sum_{r,s=0}^{\infty} (-1)^r q^s h_{m+r+s}(x) e_r^{\perp} h_s^{\perp}=
\sum_{r=0}^{\infty} h_{m+r}(x) h_r^{\perp}\big(x(q-1)\big),
\]
where $h_r$ and $e_r$ are the $r$th complete and elementary symmetric 
functions, and where the rightmost expression again uses $\la$-rings.
Alternatively, if $B(z)=B(z;x;q)$ is the vertex operator 
$B(z)=\sum_m z^m B_m$, then 
\[
B(z)(f)=f\Big(x-\frac{1-q}{z}\Big)\prod_{i\geq 1}\frac{1}{1-zx_i}.
\]
For a partition $\la=(\la_1,\dots,\la_k)$ Jing \cite{Jing91} showed that
\begin{equation}\label{Eq_Jing}
Q'_{\la}(x;q)=B_{\la_1}\cdots B_{\la_k}(1),
\end{equation}
or, equivalently, $Q'_0(x;q)=1$ and 
\begin{equation}\label{Eq_Qrec}
Q'_{\nu}(x;q)=B_m\big(Q'_{\la}(x;q)\big),
\end{equation}
where $\nu=(m,\la_1,\la_2,\dots,\la_k)$ for $m\geq \la_1$.
We note that although $B_0$ is not the identity operator, $B_0(1)=1$
so that \eqref{Eq_Jing} is true regardless of whether $l(\la)=k$ or
$l(\la)<k$. In \cite{Garsia92} Garsia expressed the $B_m$ in more 
explicit form as
\begin{equation}\label{Eq_BGarsia}
B_m(x;q)=\sum_{i=1}^n x_i^m 
\Bigg(\prod_{\substack{j=1 \\ j \neq i}}^n \frac{x_i}{x_i-x_j}\Bigg) T_{q,x_i},
\end{equation}
where $(T_{q,x_i} f)(x)=f(x_1,\dots,x_{i-1},qx_i,x_{i+1},\dots,x_n)$.
It is this representation that will be important 
in proving our hypergeometric formula for $P'_{\la}$.

\subsection{The modified Hall--Littlewood polynomial as $q$-hypergeometric 
multisum}

Define the $q$-shifted factorial $(a)_n=(a;q)_n$ indexed by an 
arbitrary integer $n$ as $(a)_n=(a)_{\infty}/(aq^n)_{\infty}$, where 
$(a)_{\infty}=(1-a)(1-aq)\cdots$. 
For $r,s\in\NN^n$, $\tau$ an integer and $x=(x_1,\dots,x_n)$ we
define the $q$-hypergeometric term
\begin{equation}\label{Eq_frs}
f_{r,s}^{(\tau)}(x;q):=
\prod_{i=1}^n \Big(x_i^{r_i} q^{\binom{r_i}{2}} \Big)^{\tau}
\prod_{i,j=1}^n \frac{(qx_i/x_j)_{r_i-r_j}}{(qx_i/x_j)_{r_i-s_j}}.
\end{equation}
Since $1/(q)_n=0$ for $n<0$, it follows that 
$f_{r,s}^{(\tau)}(x;q)=0$ unless $r_i\geq s_i$ 
for all $1\leq i\leq n$, or more succinctly, $r\supseteq s$ 
for $r$ and $s$ viewed as compositions.

\begin{theorem}\label{Thm_Qphyper} 
The modified Hall--Littlewood polynomial $P'_{\la}$ is given by
\begin{equation}\label{Eq_Qphyper}
P'_{\la}(x;q)=\sum \prod_{\ell\geq 1} f_{r^{(\ell)},r^{(\ell+1)}}^{(1)}(x;q),
\end{equation}
where the sum is over 
$r^{(1)}\supseteq r^{(2)}\supseteq \cdots\in\NN^n$ 
such that $\abs{r^{(\ell)}}=\la'_{\ell}$.
\end{theorem}
Of course, since $\abs{r^{(\ell)}}=0$ for $\ell>l(\la')=\la_1$, all
$r^{(\ell)}$ for $\ell>\la_1$ are equal to $0:=(0^n)$ and the product
$\prod_{\ell\geq 1}$ may be replaced by a finite product $\prod_{\ell=1}^m$
where $m$ is an integer such that $m\geq\la_1$.

\begin{proof}[Proof of Theorem~\ref{Thm_Qphyper}]
Throughout the proof we write $P_{\la}$, $f_{r,s}$ and $b_{\la}$ for
$P_{\la}(x;q)$, $f_{r,s}(x;q)$ and $b_{\la}(q)$.

For $\la=0$ all $r^{(\ell)}$ in \eqref{Eq_Qphyper} are equal to $0$, 
resulting in $P'_0=1$. 

It remains to show that for $\la\neq 0$ our theorem is 
consistent with the action of the $q$-Bernstein operators. 
Before we do so, we translate \eqref{Eq_Qrec} into 
a statement for $P'_{\la}$ instead of $Q'_{\la}$.
First, by \eqref{Eq_Pp}, we get
$b_{\la} B_m\big(P'_{\la}\big)=b_{\nu} P'_{\nu}$.
But, since $\nu=(m,\la_1,\dots,\la_k)$ with $m\geq \la_1$,
we have $b_{\nu}/b_{\la}=(1-q^{\la'_m+1})$.
Hence, for $m\geq 1$,
\begin{equation}\label{Eq_Prec}
B_m\big(P'_{\la}\big)
=\big(1-q^{\la'_m+1}\big) P'_{\nu}.
\end{equation}

We now compute the left-hand
side of \eqref{Eq_Prec} using the claimed expression for $P'_{\la}$.
Let $m$ be an integer such that $m\geq \la_1$.
Recalling the remark after Theorem~\ref{Thm_Qphyper},
we replace the product in \eqref{Eq_Qphyper} by $\prod_{\ell=1}^m$ and
sum over $r^{(1)}\supseteq \cdots\supseteq r^{(m)}\in\NN^n$ 
such that $\abs{r^{(\ell)}}=\la'_{\ell}$ for $\ell=1,\dots,m$, and
$r^{(m+1)}:=0$.
By a simple calculation it follows that
\[
T_{q,x_i} \big(f_{r,s}^{(\tau)}\big)=
x_i^{-\tau} f_{r+\epsilon_i,s+\epsilon_i}^{(\tau)},
\]
where $\epsilon_i$ is the $i$th standard unit vector in $\Z^n$.
Hence
\[
T_{q,x_i} \big(P'_{\la}\big)= x_i^{-m} \sum \prod_{\ell=1}^m 
f_{r^{(\ell)}+\epsilon_i,r^{(\ell+1)}+\epsilon_i}^{(1)}.
\]
Making the variable change $r^{(\ell)}\mapsto r^{(\ell)}-\epsilon_i$ 
for $\ell=1,\dots,m$ while recalling that $r^{(m+1)}:=0$, this yields
\begin{subequations}\label{Eq_TPp}
\begin{align}
T_{q,x_i} \big(P'_{\la}\big)&=x_i^{-m} 
\sum 
\bigg(\prod_{\ell=1}^{m-1} f_{r^{(\ell)},r^{(\ell+1)}}^{(1)} \bigg)
f_{r^{(m)},\epsilon_i}^{(1)} \\
&=x_i^{-m} \sum \prod_{j=1}^n \big(1-q^{r^{(m)}_j}x_j/x_i\big)
\prod_{\ell=1}^m f_{r^{(\ell)},r^{(\ell+1)}}^{(1)},
\end{align}
\end{subequations}
where the second equality follows from
\[
f_{r,\epsilon_i}^{(\tau)}(x;q)=f_{r,0}^{(\tau)}(x;q)
\prod_{j=1}^n \big(1-q^{r_j}x_j/x_i\big).
\]
Both sums in \eqref{Eq_TPp} are over 
$r^{(1)}\supseteq \cdots\supseteq r^{(m)}\in\NN^n$ 
such that $\abs{r^{(\ell)}}=\la'_{\ell}+1$ for $\ell=1,\dots,m$, and
$r^{(m+1)}:=0$. (The variable change actually leads to 
$r^{(m)}\supseteq \epsilon_i$ but this may be relaxed to
$r^{(m)}\supseteq 0$ since the summands vanish when $r_i^{(m)}=0$.)
Therefore, by \eqref{Eq_BGarsia},
\[
B_m\big( P'_{\la}\big)
=\sum \bigg(\prod_{\ell=1}^m f_{r^{(\ell)},r^{(\ell+1)}}^{(1)}\bigg)
\sum_{i=1}^n \big(1-q^{r^{(m)}_i}\big)
\prod_{\substack{j=1 \\ j\neq i}}^n \frac{x_i-q^{r^{(m)}_j}x_j}{x_i-x_j}.
\]
Recalling the summation \cite[Lemma 1.33]{Milne85b}
\[
\sum_{i=1}^n (1-y_i)
\prod_{\substack{j=1 \\ j\neq i}}^n 
\frac{x_i-y_jx_j}{x_i-x_j}=1-y_1\cdots y_n
\]
and using that $q^{\abs{r^{(m)}}}=q^{\la'_m+1}$, we finally arrive at
\begin{equation}\label{Eq_BmQp}
B_m\big( P'_{\la}\big)=(1-q^{\la'_m+1})
\sum \prod_{\ell=1}^m f_{r^{(\ell)},r^{(\ell+1)}}^{(1)},
\end{equation}
summed over $r^{(1)}\supseteq \cdots\supseteq r^{(m)}\in\NN^n$ 
such that $\abs{r^{(\ell)}}=\la'_{\ell}+1$ for $\ell=1,\dots,m$.

To complete the proof we note that if we introduce the
new partition $\nu=(m,\la_1,\la_2,\dots)$ then
$\nu'_{\ell}=\la'_{\ell}+1$ for $\ell=1,\dots,m$ 
(and $\nu'_{\ell}=\la'_{\ell}=0$ for
$\ell>m$). Hence the sum on the right of \eqref{Eq_BmQp} yields exactly 
$P'_{\nu}$, resulting in \eqref{Eq_Prec}.
\end{proof}

The hypergeometric formula \eqref{Eq_Qphyper} may be restated by 
eliminating redundant summation indices; since
$r^{(1)}\supseteq r^{(2)}\supseteq \cdots\in\NN^n$ 
such that $\abs{r^{(l)}}=\la'_l$, it follows that
$r^{(l)}=r^{(l+1)}$ if $\la'_l=\la'_{l+1}$.
But $f_{r,r}^{(1)}f_{r,s}^{(\tau)}=f_{r,s}^{(\tau+1)}$ so that
we obtain the following equivalent formulation.

\begin{corollary}\label{Cor_Qphyper2}
Let $\la'=(M_1^{\tau_1} M_2^{\tau_2}\dots M_m^{\tau_m})$
for $M_1\geq M_2\geq\cdots\geq M_m\geq 0$ and $\tau_1,\dots,\tau_m>0$. 
Then
\begin{equation}\label{Eq_qhyper_b}
P'_{\la}(x;q)=
\sum \prod_{\ell=1}^m f_{r^{(\ell)},r^{(\ell+1)}}^{(\tau_{\ell})}(x;q),
\end{equation}
where the sum is over $r^{(1)}\supseteq\cdots\supseteq r^{(m)}\in\NN^n$ 
such that $\abs{r^{(\ell)}}=M_{\ell}$, and $r^{(m+1)}:=0$.
\end{corollary}
For $m=1$ this simplifies to Milne's expression for $P'_{\la}$
indexed by a rectangular partition $\la$, as implied by equating 
(2.7) and (2.17) of \cite{Milne94}. 
To compute $P'_{\la}$ as efficiently as possible we should take
$M_1>M_2>\cdots>M_m>0$. The result, however, is true if some of the 
$M_i$ are equal and/or zero (in which case further summation indices 
may be eliminated). The advantage of the stated form is that for 
$\tau_1=\tau_2=\cdots=\tau_m=1$ we recover Theorem~\ref{Thm_Qphyper} 
provided we rename $M_i$ as $\la'_i$. 

\subsection{A Littlewood identity for modified Hall--Littlewood polynomials}

In this section we give an important application of 
Theorem~\ref{Thm_Qphyper}, key in proving our combinatorial 
character formulas.

To state our main result we first need the definition of the 
Rogers--Szeg\H{o} polynomials. For $m$ a nonnegative integer,
the $m$th Rogers--Szeg\H{o} polynomial $H_m$ is given by
\cite{Andrews76}
\begin{equation}\label{Eq_RS}
H_m(z;q)=\sum_{i=0}^m z^i\qbin{m}{i},
\end{equation}
where 
$\qbin{m}{i}$
is a $q$-binomial coefficient. 
Following \cite{Warnaar06} we extend the above to partitions by
\[
h_{\la}(z;q)=\prod_{i\geq 1} H_{m_i(\la)}(z;q).
\]

Let $\qbin{\infty}{k}:=1/(q)_k$
and let $\la_{\odd}$ denote the partition containing the
odd-sized parts of $\la$.
For example, if $\la=(6,4,3,3,2,1,1,1)$ then $\la_{\odd}=(3,3,1,1,1)$.

\begin{theorem}\label{Thm_wzfinite}
For $M=(M_1,\dots,M_m)\in\NN^m$ and $m_0(\la):=\infty$
\begin{multline}\label{Eq_Case2}
\sum_{\substack{\la \\[1.5pt] \la_1\leq 2m}}
z^{\ell(\la_{\odd})}
P'_{\la}(x;q) h_{\la_{\odd}}(w/z;q) 
\prod_{\ell=1}^m (wz)^{M_{\ell}-\la'_{2\ell-1}}
\qbin{m_{2\ell-2}(\la)}{M_{\ell}-\la'_{2\ell-1}} \\
=\sum \prod_{i=1}^n 
\big( {-}q^{1-r^{(1)}_i}w/x_i,-q^{1-r^{(1)}_i}z/x_i\big)_{r^{(1)}_i}
\prod_{\ell=1}^m  
f_{r^{(\ell)},r^{(\ell+1)}}^{(2)}(x;q),
\end{multline}
where the sum on the right is over $r^{(1)},\dots,r^{(m)}\in\NN^n$
such that $\abs{r^{(\ell)}}=M_{\ell}$, and $r^{(m+1)}:=0$.
\end{theorem}

For actual applications as well as aesthetic reasons we should sum this 
over the sequence $M$. 
To this end we introduce the generalised Rogers--Szeg\H{o} polynomial
\[
h_{\la}^{(m)}(w,z;q)=
\prod_{\substack{i=1 \\[1pt] i \text{ odd}}}^{2m-1}
z^{m_i(\la)} H_{m_i(\la)}(w/z;q)
\prod_{\substack{i=1 \\[1pt] i \text{ even}}}^{2m-1} H_{m_i(\la)}(wz;q).
\]
For example, if $\la=(6,4,3,3,2,1,1,1)$ and $m=3$
then 
\[
h_{\la}^{(2)}(w,z;q)=z^5 H_1^2(wz;q)H_2(w/z;q)H_3(w/z;q).
\]
From $H_m(z;q)=z^m H_m(z^{-1};q)$ it is easily seen that
$h_{\la}^{(m)}(w,z;q)=h_{\la}^{(m)}(z,w;q)$.
Now taking the $M$-sum in \eqref{Eq_Case2}, interchanging the 
sums over $\la$ and $M$,  
shifting $M_{\ell}\to M_{\ell}+\la'_{2\ell-1}$ and finally
performing the $M_1$-sum using the $q$-binomial theorem 
\cite[Eq.~(2.2.5)]{Andrews76}, \eqref{Eq_Case2} simplifies to 
the following identity.

\begin{corollary}[Littlewood-type identity]\label{Cor_conjectural}
Let $\abs{wz}<1$. Then
\begin{multline}\label{Eq_sumM}
\sum_{\substack{\la \\[1.5pt] \la_1\leq 2m}} 
h_{\la}^{(m)}(w,z;q) P'_{\la}(x;q) \\
=(wz)_{\infty}\sum \prod_{i=1}^n 
\big({-}q^{1-r^{(1)}_i}w/x_i,-q^{1-r^{(1)}_i}z/x_i\big)_{r^{(1)}_i}
\prod_{\ell=1}^m f_{r^{(\ell)},r^{(\ell+1)}}^{(2)}(x;q),
\end{multline}
where the sum on the right is over $r^{(1)},\dots,r^{(m)}\in\NN^n$,
and $r^{(m+1)}:=0$.
\end{corollary}
It will be convenient later to also consider \eqref{Eq_sumM} for
$m=0$. For this purpose we define $h_0^{(0)}(w,z;q)=(wz)_{\infty}$,
so that for $m=0$ both sides trivialise to $(wz)_{\infty}$.

For $z=0,-1,-q^{\pm 1},q^{\pm 1/2}$ the Rogers--Szeg\H{o} polynomial 
\eqref{Eq_RS} completely factorises. 
In terms of $h_{\la}^{(m)}(w,z;q)$ (and up to symmetry) this 
corresponds to $(w,z)=(0,z),(1,q^{1/2}),(-1,-q^{1/2}),(q^{1/2},-q^{1/2})$, 
the case $(w,z)=(1,-1)$ being ruled out for convergence reasons.
Surprisingly, it is precisely these special cases that correspond to 
characters of affine Lie algebras.

\medskip

Before proving Theorem~\ref{Thm_wzfinite} we prepare three simple
identities satisfied by the $q$-hypergeometric term $f^{(\tau)}_{r,s}(x;q)$.

\begin{proposition}
Let $N=(N_1,\dots,N_n)\in\Z^n$, $s=(s_1,\dots,s_n)\in\Z^n$ such that 
$s\subseteq N$, and let $M\geq \abs{s}$ be an integer. Then
\begin{subequations}
\begin{align}\label{Eq_Key1}
\sum_{r\in\Z^n} 
f_{N,r}^{(\tau)}(x;q)f_{r,s}^{(1)}(x;q)
&=f_{N,s}^{(\tau)}(x;q)
\prod_{i=1}^n x_i^{s_i} q^{\binom{s_i}{2}}\frac{(-x_i)_{N_i}}{(-x_i)_{s_i}}, 
\\[1mm]
\label{Eq_Key2}
\sum_{\substack{r\in\Z^n \\[1pt] \abs{r}=M}} 
f_{N,r}^{(\tau)}(x;q)f_{r,s}^{(0)}(x;q)&=f_{N,s}^{(\tau)}(x;q) 
\qbin{\abs{N}-\abs{s}}{\hspace{2pt}M\,-\abs{s}},  \\[1mm]
\label{Eq_Key3}
\sum_{r\in\Z^n} z^{\abs{r}} f_{N,r}^{(\tau)}(x;q)f_{r,s}^{(0)}(x;q)
&=z^{\abs{s}} f_{N,s}^{(\tau)}(x;q)\, H_{\abs{N}-\abs{s}}(z;q).
\end{align}
\end{subequations}
\end{proposition}
Note that in all three cases we may restrict the sum over $r$ to
$s\subseteq r\subseteq N$.

\begin{proof}
We first prove \eqref{Eq_Key1}.
Shifting the summation index $r\mapsto r+s$ and using that 
$f_{s,s}^{(1)}(x;q)= \prod_i x_i^{s_i} q^{\binom{s_i}{2}}$, we get
\[
\sum_{r\in\Z^n} 
\frac{f_{N,r+s}^{(\tau)}(x;q)f_{r+s,s}^{(1)}(x;q)}
{f_{N,s}^{(\tau)}(x;q)f_{s,s}^{(1)}(x;q)}
=\frac{(-x_i)_{N_i}}{(-x_i)_{s_i}}.
\]
Replacing $N\mapsto N+s$ followed by $x\mapsto -xq^{-\abs{N}-s}$,
and then using
\begin{equation}\label{Eq_fx}
f_{N,r+s}^{(\tau)}(x;q)=f_{s,s}^{(\tau)}(x;q) f_{N,r}^{(\tau)}(xq^s;q)
\end{equation}
and $(aq)_{n+k}=(aq)_k(aq^k)_n$,
the $s$ dependence drops out and the resulting identity 
can be recognised as Milne's terminating $q$-binomial 
theorem \cite[Theorem 5.46]{Milne97}
\[
\Phien\big(q^{-N};\text{--}\,;q,x\big)=
\prod_{i=1}^n (q^{-\abs{N}}x_i)_{N_i},
\]
where $N_1,\dots,N_n\geq 0$ and
\begin{equation}\label{Eq_Phidef}
\Phien\big(q^{-N};\text{--}\,;q,x\big):=\sum_{r\in\NN^n} 
\frac{f_{N,r}^{(0)}\big(xq^{-\abs{N}};q\big)
f_{r,0}^{(1)}\big(xq^{-\abs{N}};q\big)}
{f_{N,0}^{(0)}\big(xq^{-\abs{N}};q\big)}.
\end{equation}

To prove the second claim we proceed in almost identical fashion.
We first write \eqref{Eq_Key2} as
\[
\sum_{\substack{r\in\Z^n \\[1pt] \abs{r}=M-\abs{s}}} 
\frac{f_{N,r+s}^{(\tau)}(x;q)f_{r+s,s}^{(0)}(x;q)}
{f_{N,s}^{(\tau)}(x;q)}
=\qbin{\abs{N}-\abs{s}}{\hspace{2pt}M\,-\abs{s}}
\]
and then make substitutions
$M\mapsto M+\abs{s}$, $N\mapsto N+s$ and $x\mapsto xq^{-s}$.
By \eqref{Eq_fx} this yields
\[
\sum_{\substack{r\in\Z^n \\[1pt] \abs{r}=M}} 
\frac{f_{N,r}^{(0)}(x;q)f_{r,0}^{(0)}(x;q)}
{f_{N,0}^{(0)}(x;q)}
=\qbin{\abs{N}}{M}
\]
which again is independent of $s$.
By the easy to verify
\[
f_{r,0}^{(0)}(x;q)=(-1)^{\abs{r}}
q^{-\binom{\abs{r}}{2}}
\prod_{1\leq i<j\leq n} \frac{x_i q^{r_i}-x_j q^{r_j}}{x_i-x_j}
\prod_{i,j=1}^n q^{\binom{r_i}{2}}
\Big({-}\frac{x_i}{x_j}\Big)^{r_i}
\frac{1}{(qx_i/x_j)_{r_i}}
\]
and
\[
\frac{f_{N,r}^{(0)}(x;q)}{f_{N,0}^{(0)}(x;q)}=q^{\abs{N}\abs{r}}
\prod_{i,j=1}^n q^{-\binom{r_i}{2}}
\Big({-}\frac{x_j}{x_i}\Big)^{r_i}
(q^{-N_j}x_i/x_j)_{r_i},
\]
this is Milne's \cite[Theorem 1.49]{Milne85} 
\[
\sum_{\substack{r\in\Z^n \\[1pt] \abs{r}=M}} 
\prod_{1\leq i<j\leq n} \frac{x_i q^{r_i}-x_j q^{r_j}}{x_i-x_j}
\prod_{i,j=1}^n \frac{(a_j x_i/x_j)_{r_i}}{(qx_i/x_j)_{r_i}}
=\frac{(a_1\cdots a_n)_M}{(q)_M}
\]
for $a_{ii}\mapsto q^{-N_i}$.

Finally, \eqref{Eq_Key3} follows after multiplying
both sides of \eqref{Eq_Key2} by $z^M$ and then summing over $M$
using
\[
\sum_{M=\abs{s}}^{\infty} z^M \qbin{\abs{N}-\abs{s}}{\hspace{2pt}M\,-\abs{s}}
=z^{\abs{s}} \sum_{M=0}^{\infty} z^M \qbin{\abs{N}-\abs{s}}{M}
=z^{\abs{s}} H_{\abs{N}-\abs{s}}(z;q). \qedhere
\] 
\end{proof}

We are now prepared to prove Theorem~\ref{Thm_wzfinite}.

\begin{proof}
We will show how to transform the left-hand side of 
\eqref{Eq_Case2}---denoted below by $\LHS$---into the right-hand side.

As a first step we apply Theorem~\ref{Thm_Qphyper} 
with $\la$ a partition such that $\la_1\leq 2m$, and replace 
$(r_{2\ell-1},r_{2\ell})\mapsto (u_{\ell},v_{\ell})$
for all $\ell=1,\dots,m$. This yields
\begin{equation}\label{Ppff}
P'_{\la}(x;q)=\sum \prod_{\ell=1}^m f_{u^{(\ell)},v^{(\ell)}}^{(1)}(x;q)
f_{v^{(\ell)},u^{(\ell+1)}}^{(1)}(x;q),
\end{equation}
summed over $u^{(1)}\supseteq v^{(1)}\supseteq \cdots\supseteq 
u^{(m)}\supseteq v^{(m)}\in\NN^n$ 
such that $\abs{u^{(\ell)}}=\la'_{2\ell-1}$
and $\abs{v^{(\ell)}}=\la'_{2\ell}$ (and as usual, $u^{(m+1)}:=0$).
Also using
\[
h_{\la_{\odd}}(w/z;q)=
\prod_{\ell=1}^m H_{m_{2\ell-1}(\la)}(w/z;q),
\]
as well as $l(\la_{\odd})=\sum_{i=1}^m m_{2\ell-1}(\la)$ and
$m_i(\la)=\la'_i-\la'_{i+1}$, we obtain
\begin{multline}\label{Eq_abo}
\LHS=
\sum_{\{u^{(\ell)},v^{(\ell)}\}}
\prod_{\ell=1}^m  \bigg\{
w^{M_{\ell}-\abs{u^{(\ell)}}}
z^{M_{\ell}-\abs{v^{(\ell)}}}
\qbin{\abs{v^{(\ell-1)}}-\abs{u^{(\ell)}}}{M_{\ell}-\abs{u^{(\ell)}}} \\
\times
H_{\abs{u^{(\ell)}}-\abs{v^{(\ell)}}}(w/z;q) 
f^{(1)}_{u^{(\ell)},v^{(\ell)}}(x;q)
f^{(1)}_{v^{(\ell)},u^{(\ell+1)}}(x;q) \bigg\},
\end{multline}
where $\sum_{\{u^{(\ell)},v^{(\ell)}\}}$ is shorthand for a sum over
$u^{(1)}\supseteq v^{(1)}\supseteq \cdots\supseteq 
u^{(m)}\supseteq v^{(m)}\in\NN^n$.
In the above $\abs{v^{(0)}}$ should be interpreted as $\infty$. 
Concerning this occurrence of $\infty$ in one of the $q$-binomial coefficients, 
we remark that although $\lim_{N\to(\infty^n)}f_{N,r}(x;q)$ does not
exist, $\lim_{N\to(\infty^n)} f_{N,r}(x;q)/f_{N,s}(x;q)$ does and is given by $1$. 
In the next step of our proof we write, by abuse of notation,
\[
\sum_{\substack{r\in\Z^n \\[1pt] \abs{r}=M}} f^{(0)}_{r,s}(x;q)=\frac{1}{(q)_{M-\abs{s}}}
\]
as
\[
\sum_{\substack{r\in\Z^n \\[1pt] \abs{r}=M}} f^{(\tau)}_{(\infty^n),r}(x;q)f^{(0)}_{r,s}(x;q)=
f^{(\tau)}_{(\infty^n),s}(x;q) \qbin{\infty}{M-\abs{s}}.
\]
With this in mind we we apply \eqref{Eq_Key2} and \eqref{Eq_Key3} with $\tau=1$
to expand \eqref{Eq_abo} as
\begin{multline*}
\LHS=
\sum_{\substack{\{r^{(\ell)},s^{(\ell)},u^{(\ell)},v^{(\ell)}\} \\[1pt]
\abs{r^{(\ell)}}=M_{\ell}}}
\prod_{\ell=1}^m  \bigg\{
w^{M_{\ell}+\abs{s^{(\ell)}}-\abs{u^{(\ell)}}-\abs{v^{(\ell)}}}
z^{M_{\ell}-\abs{s^{(\ell)}}} \\ \times
f^{(0)}_{r^{(\ell)},u^{(\ell)}}(x;q)
f^{(1)}_{u^{(\ell)},s^{(\ell)}}(x;q)
f^{(0)}_{s^{(\ell)},v^{(\ell)}}(x;q)
f^{(1)}_{v^{(\ell)},r^{(\ell+1)}}(x;q) 
\bigg\},
\end{multline*}
where $
\sum_{\{r^{(\ell)},s^{(\ell)},u^{(\ell)},v^{(\ell)}\}}$ stands for
a sum over
\[
r^{(1)}\supseteq u^{(1)}\supseteq s^{(1)}\supseteq v^{(1)}\supseteq \cdots
\supseteq 
r^{(m)}\supseteq u^{(m)}\supseteq s^{(m)}\supseteq v^{(m)}\in\NN^n.
\]
By 
\begin{equation}\label{Eq_f-scalar}
f^{(\tau)}_{N,r}(ax;q)=a^{\abs{N}\tau}f^{(\tau)}_{N,r}(x;q)
\end{equation}
for $a$ a scalar, this is also
\begin{multline*}
\LHS=
\sum_{\substack{\{r^{(\ell)},s^{(\ell)},u^{(\ell)},v^{(\ell)}\} \\[1pt]
\abs{r^{(\ell)}}=M_{\ell}}}
\prod_{\ell=1}^m  \bigg\{
w^{M_{\ell}+\abs{s^{(\ell)}}}
z^{M_{\ell}-\abs{s^{(\ell)}}} \\ \times
f^{(0)}_{r^{(\ell)},u^{(\ell)}}\Big(\frac{x}{w};q\Big)
f^{(1)}_{u^{(\ell)},s^{(\ell)}}\Big(\frac{x}{w};q\Big)
f^{(0)}_{s^{(\ell)},v^{(\ell)}}\Big(\frac{x}{w};q\Big)
f^{(1)}_{v^{(\ell)},r^{(\ell+1)}}\Big(\frac{x}{w};q\Big)
\bigg\}.
\end{multline*}
By \eqref{Eq_Key1} we can now perform the sums over 
$\{u^{(\ell)}\}$ and $\{v^{(\ell)}\}$, so that 
\begin{multline*}
\LHS=
\sum_{\substack{\{r^{(\ell)},s^{(\ell)}\} \\[1pt]
\abs{r^{(\ell)}}=M_{\ell}}}
\prod_{\ell=1}^m  \bigg\{
w^{M_{\ell}+\abs{s^{(\ell)}}}
z^{M_{\ell}-\abs{s^{(\ell)}}}
f^{(0)}_{r^{(\ell)},s^{(\ell)}}\Big(\frac{x}{w};q\Big)
f^{(0)}_{s^{(\ell)},r^{(\ell+1)}}\Big(\frac{x}{w};q\Big) \\ \times
\prod_{i=1}^n \Big(\frac{x_i}{w}\Big)^{s_i^{(\ell)}} q^{\binom{s_i^{(\ell)}}{2}}
\frac{(-x_i/w)_{r^{(\ell)}_i}}{(-x_i/w)_{s^{(\ell)}_i}} \cdot
\Big(\frac{x_i}{w}\Big)^{r_i^{(\ell+1)}} q^{\binom{r_i^{(\ell+1)}}{2}}
\frac{(-x_i/w)_{s^{(\ell)}_i}}{(-x_i/w)_{r^{(\ell+1)}_i}}
\bigg\}.
\end{multline*}
By some telescoping, and the use of
\begin{equation}\label{Eq_akinv}
(a;q)_k=(-a)^k q^{\binom{k}{2}} (q^{1-k}/a)_k,
\end{equation}
\eqref{Eq_f-scalar} and $M_l=\abs{r^{(\ell)}}$,
this may be simplified to
\begin{multline*}
\LHS=
\sum_{\substack{\{r^{(\ell)},s^{(\ell)}\} \\[1pt]
\abs{r^{(\ell)}}=M_{\ell}}}
\prod_{i=1}^n (-q^{1-r^{(1)}_i}w/x_i)_{r^{(1)}_i} \\ \times
\prod_{\ell=1}^m  \bigg\{
z^{2\abs{r^{(\ell)}}}
f^{(1)}_{r^{(\ell)},s^{(\ell)}}\Big(\frac{x}{z};q\Big)
f^{(1)}_{s^{(\ell)},r^{(\ell+1)}}\Big(\frac{x}{z};q\Big)
\bigg\}.
\end{multline*}
Now using \eqref{Eq_Key1} to sum over $\{s^{(l)}\}$ results in
\begin{multline*}
\LHS=
\sum_{\substack{\{r^{(\ell)}\} \\[1pt] \abs{r^{(\ell)}}=M_{\ell}}}
\prod_{i=1}^n (-q^{1-r^{(1)}_i}w/x_i)_{r^{(1)}_i} 
\prod_{\ell=1}^m  \bigg\{
z^{2\abs{r^{(\ell)}}}
f^{(1)}_{r^{(\ell)},r^{(\ell+1)}}\Big(\frac{x}{z};q\Big) \\ \times
\prod_{i=1}^n \Big(\frac{x_i}{z}\Big)^{r^{(\ell+1)}_i} 
q^{\binom{r^{(\ell+1)}_i}{2}}
\frac{(-x_i/z)_{r^{(\ell)}_i}}{(-x_i/z)_{r^{(\ell+1)}_i}} \bigg\}.
\end{multline*}
Again using telescoping plus \eqref{Eq_f-scalar} and \eqref{Eq_akinv}
this simplifies to
\[
\LHS=
\sum_{\substack{\{r^{(\ell)}\} \\[1pt] \abs{r^{(\ell)}}=M_{\ell}}}
\prod_{i=1}^n (-q^{1-r^{(1)}_i}w/x_i,-q^{1-r^{(1)}_i}z/x_i)_{r^{(1)}_i} 
\prod_{\ell=1}^m 
f^{(2)}_{r^{(\ell)},r^{(\ell+1)}}(x;q)
\]
which is the desired right-hand side of \eqref{Eq_Case2}.
\end{proof}

\subsection{Rogers--Ramanujan-type $q$-series}
To conclude the section on modified Hall--Littlewood polynomials,
we present a conjecture which will be important in our 
discussion of Macdonald-type eta-function identities in 
Section~\ref{Sec_dedekind}.

We begin by defining a very general $q$-series of Rogers--Ramanujan
or Nahm--Zagier-type \cite{Andrews74,Nahm07,Zagier07}.
Let $C_n$ be the $n\times n$ Cartan matrix of $\A_n$, i.e.,
$(C_n^{-1})_{ab}=\min\{a,b\}-ab/(n+1)$ and let $T_m$ be the 
$m\times m$ Cartan-type matrix of the tadpole graph of $m$ vertices, 
i.e., $(T_m^{-1})_{ij}=\min\{i,j\}$. 
Then
\begin{multline}\label{Eq_Fdef}
F_{m,n}(u,w,z;q):=\sum
\prod_{a,b=1}^n \prod_{i,j=1}^m 
q^{\frac{1}{2}(C_n)_{ab} (T_m^{-1})_{ij} r_i^{(a)} r_i^{(b)}} \\
\times
\big({-}zq^{1/2-r_1^{(1)}-\cdots-r_m^{(1)}}/u\big)_{r_1^{(1)}+\cdots+r_m^{(1)}} 
\prod_{a=1}^n \big({-}u_awq^{r_m^{(a)}+1/2}\big)_{\infty}
\prod_{a=1}^n \prod_{i=1}^m \frac{u_a^{2i r_i^{(a)}}}{(q)_{r_i^{(a)}}},
\end{multline}
where the sum is over $r_i^{(a)}\in\NN$ for all $1\leq a\leq n$ and 
$1\leq i\leq m$, and $u_a:=u^{(-1)^{a-1}}$.
In particular, if $Q_{+}=\sum_{i=1}^n \NN \alpha_i$ with 
$\alpha_1,\dots,\alpha_n$ the simple roots of $\A_n$, then
\begin{multline}\label{Eq_Fded1}
F_{1,n}(u,w,z;q) \\ =\sum_{\alpha\in Q_{+}}
q^{\frac{1}{2}\ipb{\alpha}{\alpha}}
\big({-}zq^{1/2-\ipb{\alpha}{\Lambda_1}}/u\big)_{\ipb{\alpha}{\Lambda_1}} 
\prod_{a=1}^n \frac{ u_a^{2\ipb{\alpha}{\Lambda_a}}
\big({-}u_awq^{1/2+\ipb{\alpha}{\Lambda_a}}\big)_{\infty}}
{(q)_{\ipb{\alpha}{\Lambda_a}}}.
\end{multline}
Several important $q$-series arise as special cases:
$F_{1,1}(1,0,0;q)$ and $F_{1,1}(q^{1/2},0,0;q)$ are the 
Rogers--Ramanujan $q$-series,
$F_{k-1,1}(1,w,0;q)$ for $w=0$ and $w=q^{1/2}$ are the (first) 
Andrews--Gordon $q$-series \cite{Andrews74}
and its even modulus generalisation due to Bressoud \cite{Bressoud80},
and $F_{k-1,1}(1,w^{1/2},1;q)$ for $w=0$ and $w=q^{1/2}$
are the generalised G\"ollnitz--Gordon $q$-series \cite{Andrews75} and
its even modulus variant \cite{Bressoud80b}.

\begin{conjecture}\label{Con_spec}
Let $m,n\geq 1$ and $F_{m,n}(u,w,z;q)$ as defined in \eqref{Eq_Fdef}
Specialising $x=q^{1/2}(u,u^{-1},u,u^{-1},\dots)$ in the left-hand side
of \eqref{Eq_sumM} yields
\begin{equation}\label{Eq_Fdn}
\sum_{\substack{\la \\[1.5pt] \la_1\leq 2m}} 
q^{\abs{\la}/2}h_{\la}^{(m)}(w,z;q) P'_{\la}
\big(\underbrace{u,u^{-1},u,u^{-1},\dots}_{n \text{ terms}};q\big)
=F_{m,n}(u,w,z;q).
\end{equation}
\end{conjecture}
We note that for $(w,z)\neq (0,0)$ we effectively have two 
conjectures since the symmetry $F_{m,n}(u,w,z;q)=F_{m,n}(u,z,w;q)$
implied by the conjecture is not at all evident.
In the rank-$1$ case the conjecture is easily proved using
standard manipulations for $q$-hypergeometric series.
The conjecture also holds for $u=1$, $w=z=0$ and $n$ even
thanks to \eqref{Eq_FS} below, and for $u=1,q^{1/2}$, $w=z=0$ and $m=1$ 
by \cite[Theorem 4.1]{WZ12}.
The proof of that theorem only requires minor modifications
to settle the conjecture for $m=1$ and arbitrary $u$, $w$ and $z$.
\begin{theorem}\label{Thm_diseen}
Equation \eqref{Eq_Fdn} holds for $m=1$.
\end{theorem}

\begin{proof}
One possible approach would be to specialise 
$x=q^{1/2}(u,u^{-1},u,u^{-1},\dots)$ in the
$m=1$ case of Corollary~\ref{Cor_conjectural} and prove that
\begin{multline*}
(wz)_{\infty}\sum \prod_{i=1}^n 
u_i^{2r_i} q^{r_i^2} \big({-}q^{1/2-r_i}w/u_i,-q^{1/2-r_i}z/u_i\big)_{r_i} 
\prod_{i,j=1}^n \frac{(qu_i/u_j)_{r_i-r_j}}{(qu_i/u_j)_{r_i}} \\
=\sum 
\big({-}zq^{1/2-r_1}/u\big)_{r_1} 
\prod_{i=1}^n u_i^{2r_i}q^{r_i^2-r_i r_{i+1}}
\frac{({-}u_iwq^{r_i+1/2})_{\infty}}{(q)_{r_i}},
\end{multline*}
where $u_i=u^{(-1)^{i-1}}$ and $r_{n+1}:=0$.
Using standard basic hypergeometric notation, for $n=1$ this
is equivalent to the $c\to 0$ limit of Heine's transformation 
\cite[Equation (III.2)]{GR04}
\[
{_2}\phi_1\bigg[\genfrac{}{}{0pt}{}{a,b}{c};q,z\bigg]=
\frac{(c/b,bz)_{\infty}}{(c,z)_{\infty}}\,
{_2}\phi_1\bigg[\genfrac{}{}{0pt}{}{abz/c,b}{bz};q,\frac{c}{b}\bigg]
\]
with $(a,b,z)\mapsto (-q^{1/2}u/w,-q^{1/2}u/z,wz)$.
For $n>1$, however, proving the above appears rather nontrivial.

There is however a second approach based on the following formula
for modified Hall--Littlewood polynomials \cite{Kirillov00,WZ12}:
\[
P'_{\la}(x;q)=\sum 
\prod_{j\geq 1} \bigg(\frac{1}{(q)_{\mu_j^{(0)}-\mu_{j+1}^{(0)}}}
\prod_{i=1}^n 
x_i^{\mu^{(i-1)}_j-\mu^{(i)}_j}
q^{\binom{\mu^{(i-1)}_j-\mu^{(i)}_j}{2}} 
\qbin{\mu^{(i-1)}_j-\mu^{(i)}_{j+1}}{\mu^{(i-1)}_j-\mu^{(i)}_j}\bigg),
\]
where the sum is over 
$0=\mu^{(n)}\subseteq\dots\subseteq\mu^{(1)}\subseteq\mu^{(0)}=\la'$.
This can be used to compute the left-hand side of \eqref{Eq_Fdn} for $m=1$
as follows.
Introduce new summation indices $k_0,\dots,k_{n-1}$ and $r_1,\dots,r_n$ by 
\[
\mu^{(i)}=(r_{i+1}+\cdots+r_n-k_{i+1}-\cdots-k_n,
r_{i+1}+\cdots+r_n-k_i-\cdots-k_n),
\]
where $k_n:=0$. Then
\begin{multline*}
\LHS=\sum_{r\in\NN^n} \frac{\prod_{i=1}^n u^{2(-1)^{i-1} r_i} 
q^{r_i^2}}{(q)_{r_n}}
\bigg(\sum_{\ell,k_0\geq 0} (zq^{1/2-r_1}/u)^{k_0} \Big(\frac{w}{z}\Big)^{\ell}
q^{\binom{k_0}{2}} 
\qbin{r_1}{k_0}\qbin{k_0}{\ell} \bigg) \\ \times
\prod_{i=1}^{n-1}\bigg(\sum_{k_i\geq 0}
\frac{q^{k_i(k_i-r_i-r_{i+1})}}{(q)_{r_i-k_i}}\qbin{r_{i+1}}{k_i}\bigg).
\end{multline*}
The sum over $k_i$ (for $1\leq i\leq n-1)$ can be carried out by the 
$q$-Chu--Vandermonde sum \cite[Equation (3.3.10)]{Andrews76} to yield 
$q^{-r_ir_{i+1}}/(q)_{r_i}$.
Then shifting $(r_1,k_0)\mapsto (r_1+\ell,k_0+\ell)$ we can successively
sum over $k_0$ and $\ell$ by the $q$-binomial theorem, resulting in
\[
\LHS=\sum_{r\in\NN^n} \big({-}zq^{1/2-r_1}/u\big)_{r_1}
\big({-}uwq^{1/2+r_1-r_2}\big)_{\infty}
\prod_{i=1}^n \frac{u^{2(-1)^{i-1} r_i} q^{r_i^2-r_i r_{i+1}}}{(q)_{r_i}}.
\]
This is also
\begin{multline*}
\LHS=\sum_{r_1,r_3,\dots,r_n=0}^{\infty} \,
\prod_{\substack{i=1 \\ i\neq 2}}^n 
\frac{u^{2(-1)^{i-1} r_i} q^{r_i^2-r_i r_{i+1}}}{(q)_{r_i}}
\\ \times
\big({-}zq^{1/2-r_1}/u\big)_{r_1}
\big({-}uwq^{1/2+r_1}\big)_{\infty} \,
{_1}\phi_1\bigg[\genfrac{}{}{0pt}{}{-q^{1/2-r_1}/(uw)}{0};
q,-\frac{wq^{1/2-r_3}}{u}\bigg].
\end{multline*}
By ${_1}\phi_1(a;0;q,z)=(z)_{\infty}\, {_0}\phi_1(\text{--}
\hspace{0.5pt};z;q,az)$ 
\cite[Equation (III.4)]{GR04} this can be transformed into
\begin{multline*}
\LHS=\sum_{r\in\NN^n} \prod_{i=1}^n 
\frac{u^{2(-1)^{i-1} r_i} q^{r_i^2-r_i r_{i+1}}}{(q)_{r_i}} \\ \times
\big({-}zq^{1/2-r_1}/u\big)_{r_1} \big({-}uwq^{1/2+r_1}\big)_{\infty} 
\big({-}wq^{1/2+r_2-r_3}/u)_{\infty}.
\end{multline*}
We now simply keep iterating the above transformation, first on $r_3$, 
then on $r_4$ and so on, until we arrive at
\[
\LHS=\sum_{r\in\NN^n} 
\big({-}zq^{1/2-r_1}/u\big)_{r_1}
\prod_{i=1}^n 
\frac{u^{2(-1)^{i-1} r_i} q^{r_i^2-r_i r_{i+1}}}{(q)_{r_i}} \,
\big({-}u_iwq^{1/2+r_i}\big)_{\infty}.
\]
This is equivalent to \eqref{Eq_Fded1}, completing the proof.
\end{proof}

\section{The $\C_n$ Andrews transformation}\label{Sec_Cn-Andrews}

Andrews' multiple series transformation \cite{Andrews75} is one of the 
most complicated results in all of the theory of basic hypergeometric series.
It is also one of the most useful; it implies
many important partition and Rogers--Ramanujan-type identities
\cite{Andrews75} and has recently played a major role in answering
deep arithmetic questions related to the Riemann zeta function, 
see e.g., \cite{JM10,KR07,KR07b,Zudilin11}.

In this section we apply the Milne--Lilly $\C_n$ Bailey lemma
to prove a $\C_n$-analogue of Andrews' transformation. 
This result in itself is too complicated to be of much independent 
interest, but as we will see in Section~\ref{Sec_5}, 
characters of affine Lie algebras arise through 
specialisation, allowing us to prove the claims of the introduction.

\subsection{The Milne--Lilly $\C_n$ Bailey lemma}

The Bailey lemma is a standard tool in the theory of basic hypergeometric
series, see e.g., \cite{Andrews84,Andrews85,Andrews01,AAR99,Warnaar01}.
The generalisation of the Bailey machinery to the $\C_n$ (as well as 
$\A_n$) root system was developed by Milne and Lilly in a series
of papers \cite{ML92,LM93,ML95}. (Quite a different Bailey lemma
for the non-reduced root system $\mathrm{BC}_n$ was recently discovered
by Coskun \cite{Coskun}.)
We begin with the definition of a $\C_n$ Bailey pair, albeit using
a slightly different normalisation than Milne and Lilly.
Two sequences $\alpha=(\alpha_N)_{N\in\NN^n}$ and
$\beta=(\beta_N)_{N\in\NN^n}$ are said to form a $\C_n$
Bailey pair if
\begin{equation}\label{Eq_BP}
\beta_N=\sum_{0\subseteq r\subseteq N} \alpha_r
\prod_{i,j=1}^n
\frac{1}{(qx_i/x_j)_{N_i-r_j}(qx_ix_j)_{N_i+r_j}},
\end{equation}
where we remind the reader that $0\subseteq r\subseteq N$ stands for
$0\leq r_i\leq N_i$ for $i=1,\dots,n$.
The above definition may be inverted, expressing $\alpha$ in terms of $\beta$:
\begin{multline}\label{Eq_inversion}
\alpha_N=
\frac{\Delta_{\C}(xq^N)}{\Delta_{\C}(x)}
\sum_{0\subseteq r\subseteq N} \beta_r \,q^{-(n-1)\abs{r}}
\prod_{1\leq i<j\leq n}
\frac{x_iq^{r_i}-x_jq^{r_j}}{x_i-x_j}\cdot
\frac{1-x_ix_jq^{r_i+r_j}}{1-x_ix_j} \\
\times 
\prod_{i,j=1}^n \Big({-}\frac{x_i}{x_j}\Big)^{N_i-r_j} 
q^{\binom{N_i-r_j}{2}}\frac{(x_ix_j)_{N_i+r_j}}{(qx_i/x_j)_{N_i-r_j}}.
\end{multline}
The most important ingredient of the theory is the Bailey lemma,
which generates an infinite sequence of Bailey pairs from a given seed. 
Unfortunately Milne and Lilly's $\C_n$ Bailey lemma, first stated as
\cite[Equation 2.5]{ML92} and copied verbatim in \cite{LM93} and 
\cite{ML95} contains a minor typographical error in the expression for
$\beta'_N$. In the following this has been corrected.
\begin{lemma}[$\C_n$ Bailey lemma]\label{Lem_CnBailey}
If $(\alpha,\beta)$ is a $\C_n$ Bailey pair,
then so is the new pair $(\alpha',\beta')$ given by
\begin{align*}
\alpha'_N&=\alpha_N\prod_{i=1}^n\frac{(bx_i,cx_i)_{N_i}}
{(qx_i/b,qx_i/c)_{N_i}}\,\Big(\frac{q}{bc}\Big)^{N_i},\\
\beta'_N&=\sum_{0\subseteq r\subseteq N} \beta_r \, 
(q/bc)_{\abs{N}-\abs{r}}\,\Big(\frac{q}{bc}\Big)^{\abs{r}}
\prod_{i=1}^n \frac{(bx_i,cx_i)_{r_i}}
{(qx_i/b,qx_i/c)_{N_i}} \\
&\qquad\quad\times
\prod_{1\leq i<j\leq n} \frac{(qx_ix_j)_{r_i+r_j}}{(qx_ix_j)_{N_i+N_j}}
\prod_{i,j=1}^n\frac{(qx_i/x_j)_{r_i-r_j}}
{(qx_i/x_j)_{N_i-r_j}},
\end{align*}
where $b,c$ are indeterminates.
\end{lemma}

Equipped with the above lemma it is straightforward to obtain the 
$\C_n$-analogue of Andrews' transformation formula.

\begin{theorem}[$\C_n$ Andrews transformation]\label{Thm_CnAndrews}
For $m$ a nonnegative integer and $N\in\NN^n$,
\begin{align}\label{Eq_Cn-Andrews}
&\sum_{0\subseteq r\subseteq N}\frac{\Delta_{\C}(x q^r)}{\Delta_{\C}(x)}\,
\prod_{i=1}^n \Bigg[\prod_{\ell=1}^{m+1} \frac{(b_{\ell}x_i,c_{\ell}x_i)_{r_i}}
{(qx_i/b_{\ell},qx_i/c_{\ell})_{r_i}}
\Big(\frac{q}{b_{\ell}c_{\ell}}\Big)^{r_i} 
 \\ &\qquad\qquad\qquad\qquad\qquad\times 
\prod_{j=1}^n 
\frac{(q^{-N_j}x_i/x_j,x_i x_j)_{r_i}}{(q x_i/x_j,q^{N_j+1}x_ix_j)_{r_i}}\,
q^{N_jr_i} 
\Bigg] \notag \\
&=\prod_{i,j=1}^n (qx_ix_j)_{N_i}
\prod_{1\leq i<j\leq n} \frac{1}{(qx_ix_j)_{N_i+N_j}} \notag \\
&\quad\times\sum_{r^{(1)},\dots,r^{(m)}\in\NN^n}
\prod_{i,j=1}^n 
\frac{(qx_i/x_j)_{N_i}}{(qx_i/x_j)_{N_i-r_j^{(1)}}}
\prod_{\ell=1}^m f_{r^{(\ell)},r^{(\ell+1)}}^{(0)}(x;q) \notag \\ 
&\quad\qquad\times
\prod_{\ell=1}^{m+1}\Bigg[
(q/b_{\ell}c_{\ell})_{\abs{r^{(\ell-1)}}-\abs{r^{(\ell)}}} 
\Big(\frac{q}{b_{\ell}c_{\ell}}\Big)^{\abs{r^{(\ell)}}} 
\prod_{i=1}^n 
\frac{(b_{\ell}x_i,c_{\ell}x_i)_{r^{(\ell)}_i}}
{(qx_i/b_{\ell},qx_i/c_{\ell})_{r^{(\ell-1)}_i}} \Bigg] , \notag
\end{align}
where $r^{(0)}:=N$ and $r^{(m+1)}:=0$.
\end{theorem}

For $m=0$ this is Lilly and Milne's $\C_n$ analogue of 
Jackson's $_6\phi_5$ summation \cite[Theorem 2.11]{LM93} and for
$m=1$ it is Milne's $\C_n$ analogue of Watson's $q$-Whipple 
transformation 
\cite[Theorem A.3]{Milne94} (see also \cite[Theorem 6.6]{ML95}).

\begin{proof}[Proof of Theorem~\ref{Thm_CnAndrews}]
Taking $\beta_N=\delta_{N,0}=\prod_{i=1}^n \delta_{N_i,0}$
in \eqref{Eq_inversion} yields the $\C_n$ unit Bailey pair
\[
\alpha_N=
\frac{\Delta_{\C}(xq^N)}{\Delta_{\C}(x)}
\prod_{i,j=1}^n \Big({-}\frac{x_i}{x_j}\Big)^{N_i} 
q^{\binom{N_i}{2}}\frac{(x_ix_j)_{N_i}}{(qx_i/x_j)_{N_i}}
\quad\text{and}\quad \beta_N=\delta_{N,0}.
\]
Iterating this using the Bailey lemma and induction
we obtain the new Bailey pair
\begin{align*}
\alpha_N&=\frac{\Delta_{\C}(x q^N)}{\Delta_{\C}(x)}\,
\prod_{i=1}^n \Bigg[\prod_{\ell=1}^{m+1} \frac{(b_{\ell}x_i,c_{\ell}x_i)_{N_i}}
{(qx_i/b_{\ell},qx_i/c_{\ell})_{N_i}}
\Big(\frac{q}{b_{\ell}c_{\ell}}\Big)^{N_i} \\
&\qquad\qquad\qquad\qquad\qquad\qquad\times \prod_{j=1}^n 
\frac{(x_i x_j)_{N_i}}{(q x_i/x_j)_{N_i}}\,
\Big({-}\frac{x_i}{x_j}\Big)^{N_i}q^{\binom{N_i}{2}}\Bigg],\notag\\
\beta_N&=\prod_{1\leq i<j\leq n}\frac{1}{(qx_ix_j)_{N_i+N_j}}\\
&\quad\times\sum_{r^{(1)},\dots,r^{(m)}\in\NN^n}
\prod_{i,j=1}^n \frac{1}{(qx_i/x_j)_{N_i-r^{(1)}_j}}
\prod_{\ell=1}^m f_{r^{(\ell)},r^{(\ell+1)}}^{(0)}(x;q) \\
&\quad\qquad\qquad\times
\prod_{\ell=1}^{m+1}\Bigg[
(q/b_{\ell}c_{\ell})_{\abs{r^{(\ell-1)}}-\abs{r^{(\ell)}}}
\Big(\frac{q}{b_{\ell}c_{\ell}}\Big)^{\abs{r^{(\ell)}}} 
\prod_{i=1}^n 
\frac{(b_{\ell}x_i,c_{\ell}x_i)_{r^{(\ell)}_i}}
{(qx_i/b_{\ell},qx_i/c_{\ell})_{r^{(\ell-1)}_i}}\Bigg].
\end{align*}
After substitution in \eqref{Eq_BP} the claim follows. 
\end{proof}

\section{The $\C_n$ Andrews transformation and character formulas}\label{Sec_5}

Isolating the variables $b_1,c_1$, we write the $\C_n$ Andrews 
transformation \eqref{Eq_Cn-Andrews} as
\begin{equation}\label{Eq_LR}
L_N(x;b_1,c_1;b_2,\dots,c_{m+1};q)=
R_N(x;b_1,c_1;b_2,\dots,c_{m+1};q)
\end{equation}
where $L_N$ stands for the left-hand side of \eqref{Eq_Cn-Andrews} and
$R_N$ for the right-hand side.
The aim of this section is to show that \eqref{Eq_LR} 
implies Theorem~\ref{Thm_Cn-character-formula} of the introduction.
After first showing that
\[
R_m(x;b,c;q):=R_{(\infty^n)}(x;b,c;
\underbrace{\infty,\dots,\infty}_{2m \text{ times}};q)
\]
can be expressed in terms of the modified Hall--Littlewood 
polynomials $P_{\la}'$, we will prove that if  
\begin{equation}\label{Eq_Lk}
L_m(x;b,c;q):=L_{(\infty^n)}(x;b,c;
\underbrace{\infty,\dots,\infty}_{2m \text{ times}};q),
\end{equation}
then $L_m\big(x^{\pm};b,c;q\big)$ is a function which
unifies certain characters of $\C_n^{(1)}$, $\A_{2n}^{(2)}$ 
and $\D_{n+1}^{(2)}$ in their in Weyl--Kac representation.
In particular, the identity 
\begin{equation}\label{Eq_LkRk}
L_m\big(x^{\pm};b,c;q\big)=R_m\big(x^{\pm};b,c;q\big)
\end{equation}
includes \eqref{Eq_Cn-char} and \eqref{Eq_A2n2-char}
of the introduction as special limiting cases. 

\subsection{The right-hand side of the $\C_n$ Andrews transformation}

Since the right-hand side of \eqref{Eq_LR} is a rational function 
we may let $b_2,c_2,\dots,b_{m+1},c_{m+1}$ tend to infinity
for fixed $N\in\NN^n$. To then take the large $N$ limit
we need to assume that $\abs{q/b_1c_1}<1$. By an appeal to
dominated convergence this yields
\begin{multline*}
R_m(x;b,c;q)=(q/bc)_{\infty} D(x;b,c;q) \\
\times\sum_{r^{(1)},\dots,r^{(m)}\in\NN^n}
\prod_{i=1}^n 
\big(q^{1-r^{(1)}_i}/bx_i,q^{1-r^{(1)}_i}/cx_i\big)_{r^{(1)}_i}
\prod_{\ell=1}^m q^{\abs{r^{(\ell)}}}
f_{r^{(\ell)},r^{(\ell+1)}}^{(2)}(x;q),
\end{multline*}
where $r^{(m+1)}:=0$, $\abs{q/bc}<1$ and
\[
D(x;b,c;q):=\prod_{i=1}^n \frac{(qx_i^2)_{\infty}}{(qx_i/b,qx_i/c)_{\infty}} 
\prod_{1\leq i<j\leq n} (qx_ix_j)_{\infty}.
\]
If we now take \eqref{Eq_sumM},
replace $(x,w,z)\mapsto (q^{1/2}x,-q^{1/2}/b,-q^{1/2}/c)$ and
use that $f_{r,s}^{(2)}(q^{1/2}x;q)=q^{\abs{r}} f_{r,s}^{(2)}(x;q)$, 
then the right-hand side of \eqref{Eq_sumM} matches the above 
expression for $R_m(x;b,c;q)$, except for the prefactor $D(x;b,c;q)$. 
Hence, for $\abs{q/bc}<1$,
\begin{equation}\label{Eq_Rk}
R_m(x;b,c;q)=D(x;b,c;q)
\sum_{\substack{\la \\[1.5pt] \la_1\leq 2m}} q^{\abs{\la}/2}
h_{\la}^{(m)}\big({-}q^{1/2}/b,-q^{1/2}/c;q\big) P'_{\la}(x;q).
\end{equation}

\subsection{The left-hand side of the $\C_n$ Andrews transformation}

Because in our initial considerations the parameters 
$b_1,c_1,\dots,b_{m+1},c_{m+1}$
and $q$ play a passive role we suppress their dependence, writing 
$L_N(x)$ instead of $L_N(x;b_1,c_1;b_2,\dots,c_k;q)$.
To transform $L_N(x)$ into a function that resembles the Weyl--Kac character 
formula we must achieve the appropriate Weyl group symmetry. 
As will be shown below, this can be realised by doubling the rank
to $2n$ and by then reducing this back to $n$ by taking a limit in 
which $n$ distinct pairs of variables tend to $1$ as follows:
\[
\lim_{y_1\to x_1^{-1},\dots,y_n\to x_n^{-1}} 
L_{(N_1,M_1,\dots,N_n,M_n)}\big(x_1,y_1,\dots,x_n,y_n\big)
=:L_{M,N}(x).
\]
We remark that this limiting process is highly non-trivial due to the
occurrence of the denominator term $\Delta_{\C}(x)$ in the
summand of $L_N(x)$. Indeed, $\Delta_{\C}(x)$ vanishes
whenever the product of two of its variables equals $1$.
For later purposes we will also consider the following 
limit in the case of an odd number of variables:
\begin{multline*}
\lim_{y_1\to x_1^{-1},\dots,y_{n-1}\to x_{n-1}^{-1},x_n\to 1} 
L_{(N_1,M_1,\dots,N_{n-1},M_{n-1},N_n)}
\big(x_1,y_1,\dots,x_{n-1},y_{n-1},x_n\big)\\
=:\hat{L}_{M,N}(\hat{x}),
\end{multline*}
where $\hat{x}=(x_1,\dots,x_{n-1})$.

\begin{proposition}\label{Prop_lim}
For $x=(x_1,\dots,x_n)$ and $M,N\in\NN^n$,
\begin{subequations}
\begin{multline}\label{Eq_lim-Even}
L_{M,N}(x)=\sum_{r\in\Z^n}
\frac{\Delta_{\C}(x q^r)}{\Delta_{\C}(x)}\,
\prod_{i=1}^n \Bigg[\prod_{\ell=1}^{m+1} \frac{(b_{\ell}x_i,c_{\ell}x_i)_{r_i}}
{(qx_i/b_{\ell},qx_i/c_{\ell})_{r_i}}
\Big(\frac{q}{b_{\ell}c_{\ell}}\Big)^{r_i} \\ 
\times
\prod_{j=1}^n 
\frac{(q^{-N_j}x_i/x_j,q^{-M_j}x_i x_j)_{r_i}}
{(q^{M_j+1}x_i/x_j,q^{N_j+1}x_ix_j)_{r_i}}\, q^{(M_j+N_j)r_i} 
\Bigg],
\end{multline}
and for $\hat{x}=(x_1,\dots,x_{n-1})$, $x=(x_1,\dots,x_{n-1},1)$, 
$M\in\NN^{n-1}$ and $N\in\NN^n$,
\begin{multline}\label{Eq_lim-Odd}
\hat{L}_{M,N}(\hat{x})=
\sum_{r\in\Z^n} \frac{\Delta_{\B}(-xq^r)}{\Delta_{\B}(-x)} 
\prod_{i=1}^n \Bigg[
\prod_{\ell=1}^{m+1} \frac{(b_{\ell}x_i,c_{\ell}x_i)_{r_i}}
{(qx_i/b_{\ell},qx_i/c_{\ell})_{r_i}}
\Big(\frac{q}{b_{\ell}c_{\ell}}\Big)^{r_i} \\ 
\times
\prod_{j=1}^{n-1} 
\frac{(q^{-M_j}x_i x_j)_{r_i}}
{(q^{M_j+1}x_i/x_j)_{r_i}}\, q^{M_jr_i} 
\prod_{j=1}^n 
\frac{(q^{-N_j}x_i/x_j)_{r_i}}
{(q^{N_j+1}x_ix_j)_{r_i}}\, q^{N_jr_i} \Bigg].
\end{multline}
\end{subequations}
\end{proposition}
A number of remarks are in order.
First of all we note that both summands vanish unless 
$-M\subseteq r\subseteq N$, i.e., $-M_i\leq r_i\leq N_i$ for all $i$ 
(where $M_n:=N_n$ in the case of \eqref{Eq_lim-Odd}).
Moreover, if we set $M_1=\dots=M_n=0$ in \eqref{Eq_lim-Even}
we recover $L_N(x)$.
Finally we note that the series on the right of \eqref{Eq_lim-Even} 
exhibits the desired symmetry, 
in that it is invariant under the natural action of the hyperoctahedral 
group. 
For example, for $n=2$,
\begin{alignat*}{3}
&L_{(M_1,M_2),(N_1,N_2)}(x_1,x_2)
&&=L_{(M_2,M_1),(N_2,N_1)}(x_2,x_1)&&= \\
&L_{(M_1,N_2),(N_1,M_2)}(x_1,x_2^{-1})
&&=L_{(N_2,M_1),(M_2,N_1)}(x_2^{-1},x_1)&&= \\
&L_{(N_1,M_2),(M_1,N_2)}(x_1^{-1},x_2)
&&=L_{(M_2,N_1),(N_2,M_1)}(x_2,x_1^{-1})&&= \\
&L_{(N_1,N_2),(M_1,M_2)}(x_1^{-1},x_2^{-1})
&&=L_{(N_2,N_1),(M_2,M_1)}(x_2^{-1},x_1^{-1}).\!\!\!&&
\end{alignat*}
The proof of Proposition~\ref{Prop_lim} is long and technical,
and has been relegated to the appendix.

\subsection{Proof of Theorem~\ref{Thm_Cn-character-formula} and 
related results}
Recall that for $x=(x_1,\dots,x_n)$ we abbreviate
$f(x_1,x_1^{-1},\dots,x_n,x_n^{-1})$ by $f\big(x^{\pm}\big)$.
By abuse of notation, 
for $x=(x_1,\dots,x_{n-1},1)$ we also denote
$f(x_1,x_1^{-1},\dots,x_{n-1},x_{n-1}^{-1},1)$ as
$f\big(x^{\pm}\big)$
(so that in this case $f\big(x^{\pm}\big)$ should not be
interpreted as $f(x_1,x_1^{-1},\dots,x_{n-1},x_{n-1}^{-1},1,1)$).

To obtain \eqref{Eq_LkRk} in a more explicit form, we let
$b_2,c_2,\dots,b_{m+1},c_{m+1}$ tend to infinity in \eqref{Eq_lim-Even}
followed by $M,N\to (\infty^n)$, and equate the resulting expression
with \eqref{Eq_Rk} with $x\mapsto x^{\pm}$. This gives
\eqref{Eq_thm-1} below.
By a similar computation starting from \eqref{Eq_lim-Odd} we obtain 
\eqref{Eq_thm-2}.

\begin{theorem}\label{Thm_main}
Let $m$ be a nonnegative integer and $\abs{q/bc}\leq 1$. Then
the following two identities hold:
\begin{subequations}\label{Eq_thm}
\begin{multline}\label{Eq_thm-1}
\frac{1}{D\big(x^{\pm};b,c;q\big)}
\sum_{r\in\Z^n}
\frac{\Delta_{\C}(x q^r)}{\Delta_{\C}(x)}\,
\prod_{i=1}^n \frac{(bx_i,cx_i)_{r_i}}{(qx_i/b,qx_i/c)_{r_i}}
\bigg(\frac{q^{1-n}}{bc}\bigg)^{r_i} \big(x_i^2 q^{r_i}\big)^{Kr_i} \\[2mm]
=\sum_{\substack{\la \\[1.5pt] \la_1\leq 2m}} q^{\abs{\la}/2}
h_{\la}^{(m)}\big({-}q^{1/2}/b,-q^{1/2}/c;q\big) 
P'_{\la}\big(x^{\pm};q\big),
\end{multline}
where $x=(x_1,\dots,x_n)$ and $K=m+n$, and
\begin{multline}\label{Eq_thm-2}
\frac{1}{D\big(x^{\pm};b,c;q\big)}
\sum_{r\in\Z^n} \frac{\Delta_{\B}(-x q^r)}{\Delta_{\B}(-x)} 
\prod_{i=1}^n \frac{(bx_i,cx_i)_{r_i}}{(qx_i/b,qx_i/c)_{r_i}}
\bigg({-}\frac{q^{3/2-n}}{bc}\bigg)^{r_i}
\big(x_i^2q^{r_i}\big)^{Kr_i} \\[2mm]
=\sum_{\substack{\la \\[1.5pt] \la_1\leq 2m}} q^{\abs{\la}/2}
h_{\la}^{(m)}\big({-}q^{1/2}/b,-q^{1/2}/c;q\big) 
P'_{\la}\big(x^{\pm};q\big),
\end{multline}
\end{subequations}
where $x=(x_1,\dots,x_{n-1},1)$ and $K=m+n-1/2$.
\end{theorem}
Recalling that $h_0^{(0)}(w,z;q)=(wz)_{\infty}$, we note that
for $m=0$ both identities are limiting cases
of Gustafson's $\C_n^{(1)}$-analogue of Bailey's sum of a
very-well poised $_6\psi_6$ series \cite{Gustafson87}.
We also note that for $b\to\infty$ the right-hand side of 
\eqref{Eq_thm-1} and \eqref{Eq_thm-2} simplifies to
\begin{equation}\label{Eq_binf}
\sum_{\substack{\la \\[1.5pt] \la_1\leq 2m}} q^{\abs{\la}/2}
(-q^{1/2}/c)^{l(\la_{\odd})} P'_{\la}\big(x^{\pm};q\big).
\end{equation}

We now consider the various specialisations of Theorem~\ref{Thm_main}.
Noting that for $x=(x_1,\dots,x_n)$,
\[
D\big(x^{\pm};b,c;q\big)=
(q)_{\infty}^n 
\prod_{i=1}^n \frac{(qx_i^{\pm 2})_{\infty}}
{(qx_i^{\pm}/b,qx_i^{\pm}/c)_{\infty}}
\prod_{1\leq i<j\leq n} (qx_i^{\pm}x_j^{\pm})_{\infty},
\]
and recalling Lemma~\ref{Lem_Cn-WK}, it follows that in the $b,c\to\infty$ 
limit the left-hand side of \eqref{Eq_thm-1} yields
the $\C_n^{(1)}$ character \eqref{Eq_char-Cn} for $\La=m\La_0$.
(Note in particular that for this highest weight the partition $\la$ in 
Lemma~\ref{Lem_Cn-WK} is $0$ so that the symplectic Schur 
function in \eqref{Eq_Cn-char-La} trivialises to $1$.)
But when $c\to\infty$ the summand of \eqref{Eq_binf}
vanishes unless $l(\la_{\odd})=0$, i.e., unless $\la$ is even.
We thus obtain \eqref{Eq_Cn-char}.
Similarly, for $b\to\infty$ and $c\to -q^{1/2}$, and by
appeal to Lemma~\ref{Lem_A2n2-WK-1} and 
$(aq)_{\infty}/(-aq^{1/2})_{\infty}=(aq^{1/2})_{\infty}(aq^2;q^2)_{\infty}$, 
we arrive at \eqref{Eq_A2n2-char}. This completes our proof of
Theorem~\ref{Thm_Cn-character-formula}.

If we take $b\to\infty$ and $c=-1$ in \eqref{Eq_thm-1}, and
use Lemma~\ref{Lem_A2n2-WK-2} as well as 
$(a^2q)_{\infty}/(-aq)_{\infty}=(aq)_{\infty}(a^2q;q^2)_{\infty}$,
we obtain our next theorem.

\begin{theorem}\label{Thm_A2n2-char}
Let $\gfrak=\A_{2n}^{(2)}$, $\La=m\La_n$ for $m$ a nonnegative integer, and 
\[
q=\eup^{-\delta}\quad\text{and}\quad 
x_i=\eup^{-\alpha_0-\cdots-\alpha_{n-i}}.
\]
Then
\begin{equation}\label{Eq_A2n2-char2}
\eup^{-\La} \ch V(\La)=\sum_{\substack{\la \\[1.5pt] \la_1\leq 2m}}
q^{(\abs{\la}+l(\la_{\odd}))/2} P'_{\la}\big(x^{\pm};q\big).
\end{equation}
\end{theorem}

Our next result corresponds to \eqref{Eq_thm-1} for $b=-1$ and $c=-q^{1/2}$.
Then the summand on the right simplifies, since
\begin{align}\label{Eq_h-spec}
h_{\la}^{(m)}(q^{1/2},1;q)&=\prod_{i=1}^{2m-1}(-q^{1/2};q^{1/2})_{m_i(\la)}
\\ &=\frac{1}{(-q^{1/2};q^{1/2})_{\infty}}
\prod_{i=0}^{2m-1}(-q^{1/2};q^{1/2})_{m_i(\la)}, \notag
\end{align}
by $H_m(q^{1/2};q)=(-q^{1/2};q^{1/2})_m$ \cite{Warnaar06}.
If on the left we use Lemma~\ref{Lem_Dn2-WK} and the simple identity 
$(a^2q)_{\infty}/(-aq^{1/2},-aq)_{\infty}=(aq^{1/2};q^{1/2})_{\infty}$,
we obtain the following theorem.

\begin{theorem}\label{Thm_Dn2}
Let $\gfrak=\D_{n+1}^{(2)}$, $\La=2m\La_0$ for $m$ a nonnegative integer, and 
\[
q=\eup^{-\delta}
\quad\text{and}\quad 
x_i=\eup^{-\alpha_i-\cdots-\alpha_n}.
\]
Then
\begin{equation}\label{Eq_twistedD}
\eup^{-\La} \ch V(\La)=\sum_{\substack{\la \\[1.5pt] \la_1\leq 2m}}
q^{\abs{\la}} 
\bigg(\,\prod_{i=0}^{2m-1} \big({-}q\big)_{m_i(\la)}\bigg)
P'_{\la}\big(x^{\pm};q^2\big).
\end{equation}
\end{theorem}

\section{Dedekind $\eta$-function identities}\label{Sec_dedekind}

In the appendix of his paper \cite{Macdonald72} Macdonald gave his 
now famous list of identities for powers of the Dedekind $\eta$-function
$\eta(\tau)=q^{1/24}\prod_{j=1}^{\infty}(1-q^j)$,
where $q=\exp(2\pi\iup\tau)$ for $\textrm{Im}(\tau)>0$.
The simplest of his identities correspond to the non-twisted 
affine Lie algebras $\gfrak=\mathrm{X}_n^{(1)}$ and yield expansions of 
$\eta(\tau)^{\dim(\mathrm{X}_n)}$.
For example, Macdonald's formula for $\C_n^{(1)}$ generalises Jacobi's
well known identity for the third power of the $\eta$-function to
\begin{equation}\label{Eq_Cn-example} 
\eta(\tau)^{2n^2+n}=c_0
\sum q^{\frac{\|v\|^2}{4(n+1)}}
\prod_{i=1}^n v_i 
\prod_{1\leq i<j\leq n}\big(v_i^2-v_j^2\big),
\end{equation}
where $c_0=1/(1!3!\cdots(2n-1)!)$ and
where the sum is over $v\in\Z^n$ such that $v_i\equiv n-i+1\pmod{2n+2}$.

In this final section we extend many of Macdonald's identities
by specialising our character formulae.
To facilitate comparison with Macdonald's results we adopt his
definitions of $\chi_{\B}$ and $\chi_{\D}$ as given by 
\eqref{Eq_chi-B} and
\[
\chi_{\D}(v)=\prod_{1\leq i<j\leq n} 
\big(v_i^2-v_j^2\big).
\]
We also write $\chi_{\gfrak}(v/w)=\chi_{\gfrak}(v)/\chi_{\gfrak}(w)$ and
define the classical $\gfrak$-Weyl vectors $\rho_{\gfrak}$ by
\[
\rho_{\B}=(n-1/2,\dots,3/2,1/2),\quad
\rho_{\C}=(n,\dots,2,1),\quad
\rho_{\D}=(n-1,\dots,1,0).
\]
Since carrying out the required specialisations in the Weyl--Kac
formula is standard, see e.g., \cite{Macdonald72,Kac90}, we only 
list the final $\eta$-function identities below. 
For $m=0$ these correspond to Macdonald's results.
In the identities below we also give alternative 
expressions for the right-hand side as implied by
Theorem~\ref{Thm_diseen} ($m=1$) and Conjecture~\ref{Con_spec}
($m\geq 2$). This equality will be written as 
$\xlongequal{\!\!\!?_{m>2}\!\!}$.
Because in each case we have $u=1$ we will write
$F_{m,n}(w,z;q)$ for $F_{m,n}(1,w,z;q)$. 

\subsection*{Type $\C_n^{(1)}$}
If we specialise $x=(x_1,\dots,x_n)$ to $(1,\dots,1)$ in \eqref{Eq_Cn-char}
we obtain a generalisation of \eqref{Eq_Cn-example} 
(or \cite[p. 136, (6)]{Macdonald72}):
\begin{multline}\label{Eq_FS}
\frac{1}{\eta(\tau)^{2n^2+n}}
\sum_v \chi_{\B}(v/\rho)\hspace{1pt}
q^{\frac{\|v\|^2-\|\rho\|^2}{4(m+n+1)}+
\frac{\|\rho\|^2}{4(n+1)}} \\
=\sum_{\substack{\la \textup{ even} \\[1.5pt] \la_1\leq 2m}}
q^{\abs{\la}/2} P'_{\la}(\underbrace{1,\dots,1}_{2n \textup{ times}};q)
=F_{m,2n}(0,0;q),
\end{multline}
where $\rho=\rho_{\C}$, $v\in\Z^n$ such that $v\equiv \rho\pmod{2m+2n+2}$
and $m\geq 0$.
The equality between the first and last expression was proved
by Feigin and Stoyanovsky \cite{FS94} ($n=1$) 
and Stoyanovsky \cite{Stoyanovsky98} ($n>1$).
The implied equality between the two expressions in the second line
proves Conjecture~\ref{Con_spec} for $n$ even, $u=1$ and $w=z=0$.

\subsection*{Type $\A_{2n}^{(2)}$ (or affine $\BC_n$)}
If we specialise $x=(x_1,\dots,x_n)$ to $(1,\dots,1)$ in \eqref{Eq_A2n2-char2}
we obtain a generalisation of \cite[page 138, (6a)]{Macdonald72}:
\begin{multline*}
\frac{\eta(2\tau)^{2n}}{\eta(\tau)^{2n^2+3n}}
\sum_v \chi_{\B}(v/\rho) \hspace{1pt}
q^{\frac{\|v\|^2-\|\rho\|^2}{2(2m+2n+1)}+
\frac{\|\rho\|^2}{2(2n+1)}} \\
=\sum_{\substack{\la \\[1.5pt] \la_1\leq 2m}} 
q^{(\abs{\la}+l(\la_{\odd}))/2}
P'_{\la}(\underbrace{1,\dots,1}_{2n \textup{ times}};q)
\xlongequal{\!\!\!?_{m>2}\!\!}
F_{m,2n}(0,q^{1/2};q),
\end{multline*}
where $\rho=\rho_{\B}$ and $v\in(\Z/2)^n$ such that 
$v\equiv \rho\pmod{2m+2n+1}$.

If we specialise $x=(x_1,\dots,x_n)$ to $(1,\dots,1)$ in \eqref{Eq_A2n2-char}
we obtain a generalisation of \cite[p. 138, (6b)]{Macdonald72}:
\begin{multline*}
\frac{1}{\eta(\tau/2)^{2n}\eta(2\tau)^{2n} \eta(\tau)^{2n^2-3n}}
\sum_v \chi_{\B}(v/\rho) \hspace{1pt}
q^{\frac{\|v\|^2-\|\rho\|^2}{2(2m+2n+1)}+\frac{\|\rho\|^2}{2(2n+1)}} \\
=\sum_{\substack{\la \\[1.5pt] \la_1\leq 2m}}
q^{\abs{\la}/2} P'_{\la}(\underbrace{1,\dots,1}_{2n \textup{ times}};q) 
\xlongequal{\!\!\!?_{m>2}\!\!}
F_{m,2n}(0,1;q),
\end{multline*}
where $\rho=\rho_{\C}$ and $v\in\Z^n$ such that $v\equiv \rho\pmod{2m+2n+1}$.

If we let $b,c\to\infty$ in \eqref{Eq_thm-2} 
and then specialise $x=(x_1,\dots,x_{n-1},1)$ to $(1,\dots,1)$
we obtain a generalisation of \cite[page 138, (6c)]{Macdonald72}:
\begin{multline}\label{Eq_RR}
\frac{1}{\eta(\tau)^{2n^2-n}} 
\sum_v (-1)^{\abs{v}-\abs{\rho}}
\chi_{\D}(v/\rho) \hspace{1pt}
q^{\frac{\|v\|^2-\|\rho\|^2}{2(2m+2n+1)}+
\frac{\|\rho\|^2}{2(2n+1)}} \\
=\sum_{\substack{\la \text{ even} \\[1.5pt] \la_1\leq 2m}} q^{\abs{\la}/2}
P'_{\la}(\underbrace{1,\dots,1}_{2n-1 \textup{ times}};q)
\xlongequal{\!\!\!?_{m>2}\!\!}
F_{m,2n-1}(0,0;q),
\end{multline}
where $\rho=\rho_{\B}$ and $v$ is summed over $(\Z/2)^n$ such that 
$v\equiv \rho\pmod{2m+2n+1}$.
By $P'_{2\la}(x;q)=x^{2\abs{\la}} q^{2n(\la)}/b_{\la}(q)$
it follows that for $n=1$ the two expressions on the second line
are identically the same and (after replacing $m$ by $k-1$)
are given by the famous Rogers--Ramanujan--Andrews--Gordon 
series \cite{Andrews74,Gordon61}
\[
\sum_{n_1,\dots,n_{k-1}\geq 0} \frac{q^{N_1^2+\cdots+N_{k-1}^2}}
{(q)_{n_1}\cdots(q)_{n_{k-1}}},
\]
where $N_i=n_i+\cdots+n_{k-1}$.
Of course, by the Jacobi triple product identity
the left hand side for $n=1$ can be written in the familiar
product form
\[
\frac{(q^k,q^{k+1},q^{2k+1};q^{2k+1})_{\infty}}{(q)_{\infty}}.
\]
We may thus view \eqref{Eq_RR} as an $\A_{2n}^{(2)}$ analogue of these
famous $q$-series identities. 
In \cite[Conjecture 1.1 and Theorem 1.2]{WZ12} 
the equality between the left-most and right-most expressions
in \eqref{Eq_RR} was conjectured and proved for $m=1$.
The connection between the Rogers--Ramanujan partition identities
and the representation theory of Kac--Moody algebras is 
certainly not new, and we refer the interested reader to
\cite{Capparelli96,LM78,LM78b,LW78,LW82,LW84,Kac90,MP87,MP99} 
and references therein.

\subsection*{Type $\B_n^{(1)}$}

If we set $b=-1$, $c=-q^{1/2}$ in \eqref{Eq_thm-2}
and then specialise $x=(x_1,\dots,x_{n-1},1)$ to $(1,\dots,1)$, we obtain a 
generalisation of \cite[p. 135, (6c)]{Macdonald72}:
\begin{align*}
& \frac{1}{\eta(\tau/2)^{2n}\eta(\tau)^{2n^2-3n}} 
\sum_v (-1)^{\abs{v}-\abs{\rho}} \chi_{\D}(v/\rho) \hspace{1pt}
q^{\frac{\|v\|^2-\|\rho\|^2}{2(2m+2n-1)}+\frac{\|\rho\|^2}{2(2n-1)}} \\
& \qquad = \sum_{\substack{\la \\[1.5pt] \la_1\leq 2m}} q^{\abs{\la}/2} 
\bigg(\,\prod_{i=0}^{2m-1} \big({-}q^{1/2};q^{1/2}\big)_{m_i(\la)}\bigg)
P'_{\la}(\underbrace{1,\dots,1}_{2n-1 \textup{ times}};q) \\
& \qquad \xlongequal{\!\!\!?_{m>2}\!\!}
(-q^{1/2};q^{1/2})_{\infty} \, F_{m,2n-1}(q^{1/2},1;q),
\end{align*}
where $\rho=\rho_{\D}$, $v\in\Z^n$ such that $v\equiv \rho\pmod{2m+2n-1}$ 
and $m_0(\la):=\infty$.
The second equality assumes $m\geq 1$.

\subsection*{Type $\A_{2n-1}^{(2)}$ (or $\B_n^{\vee}$)}
If we let $b\to\infty$, $c\to -1$ in \eqref{Eq_thm-2}
and then specialise $x=(x_1,\dots,x_{n-1},1)$ to $(1,\dots,1)$
we obtain a generalisation of \cite[page 136 (6b)]{Macdonald72}
\begin{multline}\label{Eq_Atwisted-odd}
\frac{\eta(2\tau)^{2n-1}}{\eta(\tau)^{2n^2+n-1}}
\sum (-1)^{\frac{\abs{v}-\abs{\rho}}{2(m+n)}}
\chi_{\D}(v/\rho) \hspace{1pt}
q^{\frac{\|v\|^2-\|\rho\|^2}{4(m+n)}+\frac{\|\rho\|^2}{4n}} \\
=\sum_{\substack{\la \\[1.5pt] \la_1\leq 2m}} 
q^{(\abs{\la}+l(\la_{\odd}))/2} 
P'_{\la}(\underbrace{1,\dots,1}_{2n-1 \textup{ times}};q)
\xlongequal{\!\!\!?_{m>2}\!\!} F_{m,2n-1}(0,q^{1/2};q),
\end{multline}
where $\rho=\rho_{\D}$, $v\in\Z^n$ such that 
$v\equiv \rho\pmod{2m+2n}$.
A somewhat different generalisation of the same $\eta$-function identity
arises if we take $b=-c=1$ in \eqref{Eq_thm-1}, then use \cite{Warnaar06}
\[
h_{\la}^{(m)}(-q^{1/2},q^{1/2};q)=
\begin{cases} \displaystyle
q^{l(\la_{\odd})/2} \prod_{i=1}^{2m-1}(q;q^2)_{\lceil m_i(\la)/2\rceil}
& \text{for $m_{2i-1}(\la)$ even} \\[2mm]
0 & \text{otherwise},
\end{cases}
\]
and $h_0^{(0)}(-q^{1/2},q^{1/2};q)=
(-q)_{\infty}=(q^2;q^2)_{\infty}/(q)_{\infty}$, and finally specialise
$x=(x_1,\dots,x_n)$ to $(1,\dots,1)$. Then
\begin{align*}
\LHS\eqref{Eq_Atwisted-odd}&=
\sum_{\substack{\la \\[1.5pt] \la_1\leq 2m \\[1.5pt] (\la_{\odd})'
\text{ is even}}} q^{(\abs{\la}+l(\la_{\odd}))/2}
\bigg(\,\prod_{i=0}^{2m-1}
(q;q^2)_{\lceil m_i(\la)/2\rceil}\bigg)
P'_{\la}(\underbrace{1,\dots,1}_{2n \textup{ times}};q) \\
&\xlongequal{\!\!\!?_{m>2}\!\!} (q;q^2)_{\infty} 
F_{m,2n}(-q^{1/2},q^{1/2};q),
\end{align*}
where $m_0(\la):=\infty$ and the second equality assumes $m\geq 1$.

\subsection*{Type $\D_{n+1}^{(2)}$ (or $\C_n^{\vee}$)}
If we specialise $x=(x_1,\dots,x_n)$ to $(1,\dots,1)$ in 
\eqref{Eq_twistedD} we obtain a generalisation of
\cite[page 137, (6a)]{Macdonald72}:
\begin{multline}\label{Eq_p1376a}
\frac{1}{\eta(\tau)^{2n+1}\eta(2\tau)^{2n^2-n-1}}
\sum_v \chi_{\B}(v/\rho) \hspace{1pt}
q^{\frac{\|v\|^2-\|\rho\|^2}{2(m+n)}+\frac{\|\rho\|^2}{2n}} \\
=\sum_{\substack{\la \\[1.5pt] \la_1\leq 2m}} q^{\abs{\la}} 
\bigg(\,\prod_{i=0}^{2m-1} ({-}q)_{m_i(\la)}\bigg)
P'_{\la}(\underbrace{1,\dots,1}_{2n \textup{ times}};q^2)
\xlongequal{\!\!\!?_{m>2}\!\!} (-q)_{\infty} F_{m,2n}(q,1;q^2),
\end{multline}
where $\rho=\rho_{\B}$, $v\in(\Z/2)^n$ such that $v\equiv \rho\pmod{2m+2n}$
and second equality assumes $m\geq 1$.

Finally, if we let $b\to\infty$ and $c=-q^{1/2}$ in \eqref{Eq_thm-2},
then specialise $x=(x_1,\dots,x_{n-1},1)=(1,\dots,1)$ and replace
$q\mapsto q^2$ we obtain
\begin{multline*}
\frac{1}{\eta(\tau)^{2n-1}\eta(4\tau)^{2n-1}\eta(2\tau)^{2n^2-5n+2}}
\sum_v (-1)^{\frac{\abs{v}-\abs{\rho}}{2(m+n)}}
\chi_{\D}(v/\rho) \hspace{1pt}
q^{\frac{\|v\|^2-\|\rho\|^2}{2(m+n)}+
\frac{\|\rho\|^2}{2n}} \\
=\sum_{\substack{\la \\[1.5pt] \la_1\leq 2m}} 
q^{\abs{\la}} P'_{\la}(\underbrace{1,\dots,1}_{2n-1 \textup{ times}};q^2)
\xlongequal{\!\!\!?_{m>2}\!\!} F_{m,2n-1}(0,1;q^2),
\end{multline*}
with $v$ as in \eqref{Eq_p1376a}.
For $m=0$ (and after replacing $q$ by $-q$) we recover 
\cite[page 137, (6b)]{Macdonald72}. For $m>0$ the above 
should be viewed as a generalisation of Andrews' generalised
G\"ollnitz--Gordon $q$-series \cite{Andrews75}.

\medskip

To conclude this section we remark that Leininger and Milne employed multiple 
basic hypergeometric series for $\mathrm{A}_n$ (as opposed to the $\C_n$ series 
used in this paper) to derive other infinite families of identities for 
powers of the $\eta$-function, see \cite{Leininger97}, \cite[Theorem 2.4]{LM99}
and \cite[Theorems 2.3 and 3.2]{LM99b}.

\section{Concluding remarks}\label{Sec_Conclusion}
We end the paper with some comments in response to two questions raised
by one of the referees.

The first question asked why our results do not include
combinatorial character formulas for what is perhaps the simplest affine
Lie algebra, $\A_{n-1}^{(1)}$.
Using the Milne--Lilly Bailey lemma for $\A_{n-1}$ \cite{ML92,ML95}
it is indeed possible to prove an $\A_{n-1}$ counterpart of the 
$\C_n$ Andrews transformation of Theorem~\ref{Thm_CnAndrews}.
Specialising sufficiently many of the free parameters, the right-hand side
of this transformation can again be expressed in terms of modified
Hall--Littlewood polynomials. Unfortunately, we have been unable to 
recognise (or rewrite) the left-hand side as the Weyl--Kac expression
for $\ch V(\La)$ where $\gfrak=\A_{n-1}^{(1)}$ and $\La$ is an appropriately 
chosen highest weight.
However, recently in \cite[Section 4]{GOW15} Griffin, Ono and the second 
author used Corollary~\ref{Cor_Qphyper2} to prove
a formula for characters of $\mathrm{A}_{n-1}^{(1)}$ 
of highest weight $\La=(m-k)\Lambda_0+k\Lambda_1$ 
in terms of modified Hall--Littlewood polynomials.
This formula is somewhat different in nature from the 
identities of Theorem~\ref{Thm_Cn-character-formula} 
in that it involves a limit. For example, 
when $k=0$ it takes the form
\[
\eup^{-\Lambda} \ch V(\Lambda)=
\lim_{r\to\infty}
q^{-mn\binom{r}{2}} \frac{Q'_{(m^{nr})}(x;q)}
{(x_1\cdots x_n)^{mr}},
\]
where $q=\eup^{-\alpha_0-\alpha_1-\cdots-\alpha_{n-1}}$ and
$x_i/x_{i+1}=\eup^{-\alpha_i}$ for $1\leq i\leq n-1$.
For $m=1$ this is Kirillov's formula \cite{Kirillov00} 
for the basic representation of $\A_{n-1}^{(1)}$.

The second question concerned the possibility of simpler
proofs of the combinatorial character formulas using either 
representation-theoretic ideas (utilising, for example,
the connection between affine Demazure characters and Macdonald polynomials 
\cite{Ion03,Sanderson00}) or combinatorial methods.
In fact, Rains and the second author have recently developed an 
alternative, more conceptual approach in \cite{RW15}.
In particular, using Macdonald--Koornwinder theory 
\cite{Macdonald00,Macdonald03,Koornwinder92} and virtual 
Koornwinder integrals \cite{Rains05,RV07}, we show that 
Theorems~\ref{Thm_Cn-character-formula}, \ref{Thm_A2n2-char}
and \eqref{Thm_Dn2} as well as additional identities follow 
by specialising decomposition or branching formulas for Hall--Littlewood 
polynomials of type $R$ into Hall--Littlewood polynomials of type $\A$. 
The results of \cite{RW15} still depend crucially
on Proposition~\ref{Prop_lim} of this paper but do not rely on the
$\C_n$ Bailey lemma.

\appendix

\section{Proof of Proposition~\ref{Prop_lim}}

Before proving the proposition we prepare a key lemma.
For $p$ an integer such that $0\leq p\leq n$, let 
$M=(M_1,\dots,M_p)\in\NN^p$,
$N=(N_1,\dots,N_n)\in\NN^n$ and $r\in\Z^n$, and define
\begin{subequations}\label{Eq_fkappap}
\begin{multline}
L_{M,N;r}^{(p)}(x):=
\frac{\Delta_{\C}(x q^r)}{\Delta_{\C}(x)}\,
\prod_{i=1}^n \Bigg[\prod_{\ell=1}^{m+1} \frac{(b_{\ell}x_i,c_{\ell}x_i)_{r_i}}
{(qx_i/b_{\ell},qx_i/c_{\ell})_{r_i}}
\Big(\frac{q}{b_{\ell}c_{\ell}}\Big)^{r_i} \\ 
\times
\prod_{j=1}^n 
\frac{(q^{-N_j}x_i/x_j,q^{-M_j}x_i x_j)_{r_i}}
{(q^{M_j+1}x_i/x_j,q^{N_j+1}x_ix_j)_{r_i}}\, q^{(M_j+N_j)r_i} 
\Bigg],
\end{multline}
and
\begin{equation}
L_{M,N}^{(p)}(x):=\sum_{r_1=-M_1}^{N_1}\cdots \sum_{r_n=-M_n}^{N_n}
L_{M,N;r}^{(p)}(x),
\end{equation}
\end{subequations}
where $M_{p+1}=\cdots=M_n:=0$. 
Recalling that $L_N(x)$ denotes the left-hand side of \eqref{Eq_Cn-Andrews},
we note that
\begin{equation}\label{Eq_f0}
L_N(x)=L_{\text{--},N}^{(0)}(x).
\end{equation}
We further observe that $L_{M,N}^{(n)}(x)$ coincides with the
expression for $L_{M,N}(x)$ as claimed in \eqref{Eq_lim-Even}.

Given $x=(x_1,\dots,x_n)$ we set 
$x^{(i)}=(x_1,\dots,x_{i-1},x_{i+1},\dots,x_n)$.
\begin{lemma}\label{Lem_lim}
Let $M=(M_1,\dots,M_{p-1})$ and $M'=(M_1,\dots,M_{p-1},N_{p+2})$.
For $1\leq p\leq n-1$
\[
\lim_{x_{p+1}\to x_p^{-1}} L_{M,N}^{(p-1)}(x)=L_{M',N^{(p+1)}}^{(p)}
\big(x^{(p+1)}\big).
\]
\end{lemma}

\begin{proof}
Let us first focus on the numerator and denominator terms of 
$L_{M,N}^{(p-1)}(x)$ that vanish when $x_{p+1}\to 1/x_p$.
By $\prod_{i=1}^n \prod_{j=p}^n (x_i x_j)_{r_i}$
the numerator contains the factor
$(x_p x_{p+1})_{r_p}(x_p x_{p+1})_{r_{p+1}}$,
which in turn results in a factor $(1-x_p x_{p+1})^2$ 
if $r_p$ and $r_{p+1}$ are both positive, 
$1-x_p x_{p+1}$ if only one of these is positive and 
$1$ if both are zero. From $\Delta_{\C}(xq^r)/\Delta_{\C}(x)$
we pick up the contribution
\[
\frac{1-x_p x_{p+1} q^{r_p+r_{p+1}}}{1-x_p x_{p+1}},
\]
which is $1$ if both $r_p$ and $r_{p+1}$ are zero,
but leads to a factor $(1-x_p x_{p+1})$ in the denominator
if (at least) one of $r_p,r_{p+1}$ is positive.
As a result, $L_{M,N;r}^{(p-1)}(x)$ vanishes in the limit 
$x_{p+1}\to 1/x_p$ unless one of $r_p,r_{p+1}$ is zero.

It is now a somewhat tedious, but elementary exercise to show that
\[
\lim_{x_{p+1}\to x_p^{-1}} 
\Big(L_{M,N;r}^{(p-1)}(x)\big|_{r_{p+1}=0}\Big)
=L_{M',N^{(p+1)};r^{(p+1)}}^{(p)}\big(x^{(p+1)}\big),
\]
where $r^{(i)}:=(r_1,\dots,r_{i-1},r_{i+1},\dots,r_n)$.
Again elementary, although now requiring
\begin{equation}\label{Eq_min}
\frac{(a)_{-n}}{(b)_{-n}}=\frac{(q/b)_n}{(q/a)_n} 
\Big(\frac{b}{a}\Big)^n,
\end{equation}
is to show that
\[
\lim_{x_{p+1}\to x_p^{-1}} 
\Big(L_{M,N;r}^{(p-1)}(x)\big|_{r_p=0} \Big)
=L_{M',N^{(p+1)};\hat{r}^{(p)}}^{(p)}\big(x^{(p+1)}\big),
\]
where $\hat{r}^{(i)}:=(r_1,\dots,r_{i-1},-r_{i+1},r_{i+2},\dots,r_n)$.
Consequently, 
\begin{align*}
\lim_{x_{p+1}\to x_p^{-1}} 
L_{M,N}^{(p-1)}(x)
&=
\sum_{\substack{-M_i \leq r_i\leq N_i \\[1pt]
i=1,\dots,n \\[1pt] i\neq p,p+1}}
\Bigg(\, \sum_{\substack{r_p=0 \\ r_{p+1}=0}}^{N_p}
+\sum_{\substack{r_{p+1}=1 \\ r_p=0}}^{N_{p+1}} \,\Bigg) 
\lim_{x_{p+1}\to x_p^{-1}} L_{M,N;r}^{(p-1)}(x) \\
&=\sum_{\substack{-M_i \leq r_i\leq N_i \\[1pt]
i=1,\dots,n \\[1pt] i\neq p,p+1}}
\Bigg(\,
\sum_{r_p=0}^{N_p} L_{M',N^{(p+1)};r^{(p+1)}}^{(p)}\big(x^{(p+1)}\big) \\
& \qquad\qquad\qquad\quad+
\sum_{r_{p+1}=1}^{N_{p+1}}
L_{M',N^{(p+1)};\hat{r}^{(p)}}^{(p)}\big(x^{(p+1)}\big)\Bigg),
\end{align*}
where $M_{p+2}=\cdots=M_n:=0$.
Renaming the summation index $r_{p+1}$ as $-r_p$, this yields
\[
\lim_{x_{p+1}\to x_p^{-1}} L_{M,N}^{(p-1)}(x)
=\sum_{\substack{-M'_i \leq r_i\leq N_i \\[1pt]
i=1,\dots,n \\[1pt] i\neq p+1}}
L_{M',N^{(p+1)};r^{(p+1)}}^{(p)}\big(x^{(p+1)}\big) \\
=L_{M',N^{(p+1)}}^{(p)}\big(x^{(p+1)}\big),
\]
where $M'_{p+1}=\cdots=M'_n:=0$.
\end{proof}

Equipped with Lemma~\ref{Lem_lim}, the proof of 
Proposition~\ref{Prop_lim} is straightforward.
\begin{proof}
According to Lemma~\ref{Lem_lim} 
\[
\lim_{y_p\to x_p^{-1}} L_{M,N}^{(p-1)}
(x_1,\dots,x_p,y_p,x_{p+1},\dots,x_n) 
=L_{M',N^{(p+1)}}^{(p)}(x).
\]
Iterating this equation and recalling \eqref{Eq_f0} gives
\begin{multline*}
\lim_{y_1\to x_1^{-1},\dots,y_p\to x_p^{-1}} 
L_{(N_1,M_1,\dots,N_p,M_p,N_{p+1},\dots,N_n)}
(x_1,y_1,\dots,x_p,y_p,x_{p+1},\dots,x_n) \\
=L_{M,N}^{(p)}(x).
\end{multline*}
Recalling the remark made immediately after \eqref{Eq_f0}
this yields \eqref{Eq_lim-Even} when $p=n$.
If $p=n-1$, however, we obtain
\begin{multline*}
\lim_{y_1\to x_1^{-1},\dots,y_{n-1}\to x_{n-1}^{-1}} 
L_{(N_1,M_1,\dots,N_{n-1},M_{n-1},N_n)}
(x_1,y_1,\dots,x_{n-1},y_{n-1},x_n) \\
=\sum_{-M\subseteq r\subseteq N}
\frac{\Delta_{\C}(x q^r)}{\Delta_{\C}(x)}\,
\prod_{i=1}^n \Bigg[\prod_{\ell=1}^{m+1} \frac{(b_{\ell}x_i,c_{\ell}x_i)_{r_i}}
{(qx_i/b_{\ell},qx_i/c_{\ell})_{r_i}}
\Big(\frac{q}{b_{\ell}c_{\ell}}\Big)^{r_i} \\ 
\times
\prod_{j=1}^n 
\frac{(q^{-N_j}x_i/x_j,q^{-M_j}x_i x_j)_{r_i}}
{(q^{M_j+1}x_i/x_j,q^{N_j+1}x_ix_j)_{r_i}}\, q^{(M_j+N_j)r_i}
\Bigg],
\end{multline*}
where $M_n:=0$.
Letting $x_n$ tend to $1$,
treating the $r_n=0$ and $r_n>0$ cases 
of the summand separately, results in
\begin{multline*}
\hat{L}_{M,N}(\hat{x})=
\sum_{-M\subseteq r\subseteq N} u_{r_n}
\frac{\Delta_{\B}(-x q^r)}{\Delta_{\B}(-x)}
\prod_{i=1}^n \Bigg[
\prod_{\ell=1}^{m+1} \frac{(b_{\ell}x_i,c_{\ell}x_i)_{r_i}}
{(qx_i/b_{\ell},qx_i/c_{\ell})_{r_i}}
\Big(\frac{q}{b_{\ell}c_{\ell}}\Big)^{r_i} \\ 
\times
\prod_{j=1}^{n-1} 
\frac{(q^{-M_j}x_i x_j)_{r_i}}{(q^{M_j+1}x_i/x_j)_{r_i}}\, q^{M_jr_i} 
\prod_{j=1}^n 
\frac{(q^{-N_j}x_i/x_j)_{r_i}}
{(q^{N_j+1}x_ix_j)_{r_i}}\, q^{N_jr_i} 
\Bigg],
\end{multline*}
where $x=(x_1,\dots,x_{n-1},1)$ (so that $x_n:=1$), $M_n:=0$, $u_0=1$ 
and $u_i=2$ for $1\leq i\leq N_n$.
Using \eqref{Eq_min} and the fact that for $x_n=1$
\[
\frac{\Delta_{\B}(-xq^r)}{\Delta_{\B}(-x)} \bigg|_{r_n\mapsto -r_n}=
q^{-(2n-1)r_n} \frac{\Delta_{\B}(-xq^r)}{\Delta_{\B}(-x)}, 
\]
this can be rewritten in exactly the same functional form as the above
but now with $M_n:=N_n$ and $u_i=1$ for all $-M_n\leq i\leq N_n$.
\end{proof}

\bibliographystyle{amsplain}

\end{document}